\documentclass[reqno, english, a4paper, 12pt]{amsart}
\usepackage{babel}	
\usepackage[utf8]{inputenc} 
\usepackage[T1]{fontenc}

\usepackage{cite}

\usepackage{relsize}%for smaller and larger math symbols	

\usepackage{mathtools} % \mathclap , amsmath < mathtools
\usepackage{amssymb, amsthm, amsfonts, amscd} 
\usepackage{mathrsfs} % pre \mathscr 
\usepackage{dsfont} % pre  \mathds{1} 
\usepackage[normalem]{ulem} % \sout
\usepackage{centernot} % \centernot command	

\usepackage{enumerate}
\usepackage[margin=25mm]{geometry}	
	 
\usepackage{graphicx}
\usepackage{xcolor} %farba

\usepackage{hyperref} %link na preklik
\usepackage{verbatim}	%prostredie na komenty
\usepackage{float} %pozicia obrazkov
\usepackage{tikz}
\usetikzlibrary{trees}
\usetikzlibrary{intersections}
\usepackage{tikz-cd}
\usepackage{ifthen}

\usepackage{caption}
\usepackage{subcaption}

\newtheorem{thm}{Theorem}
\newtheorem{cor}[thm]{Corollary}
\newtheorem{lem}[thm]{Lemma}
\newtheorem{prop}[thm]{Proposition}
\newtheorem*{Q}{Question}
\newtheorem*{claim*}{Claim}

\theoremstyle{definition}
\newtheorem{definition}[thm]{Definition}
\newtheorem*{definition*}{Definition}

\newtheorem*{note*}{Note}

\newtheorem*{example*}{Example}

\theoremstyle{remark}
\newtheorem{rem}[thm]{Remark}

\numberwithin{equation}{section}

\renewcommand{\epsilon}{\varepsilon}
\renewcommand{\phi}{\varphi}

\newcommand{\hpol}{h_{\textnormal{pol}}}
\newcommand{\htop}{{h_{\textnormal{top}}}}

\DeclareMathOperator{\sep}{sep}
\newcommand{\spanning}{\mathop{\rm span}}

\DeclareMathOperator{\Fix}{Fix}
\DeclareMathOperator{\Per}{Per}	
\DeclareMathOperator{\Type}{Type}
\DeclareMathOperator{\Int}{Int}

\renewcommand{\a}{\underline{a}}

\newcommand{\G}{\mathcal{G}}
\newcommand{\Nzero}{{\mathbb{N}_0}}
\newcommand{\N}{\mathbb{N}}

\newcommand{\Lg}{\mathscr{L}}%\L obsadene polskym L

\newcommand{\conv}{\mathop{\rm conv}}
\newcommand{\NDW}{\mathop{\rm NDW}}

\newcommand{\I}{\mathcal{I}}

\newcommand{\U}{\mathcal{U}}

\DeclareMathOperator{\Orb}{Orb}
\DeclareMathOperator{\source}{source}

\makeatletter
\def\l@section{\@tocline{1}{0pt}{1pc}{}{}}
\def\l@subsection{\@tocline{2}{0pt}{1pc}{4.6em}{}}
\def\l@subsubsection{\@tocline{3}{0pt}{1pc}{7.6em}{}}
\renewcommand{\tocsection}[3]{%
	\indentlabel{\@ifnotempty{#2}{\makebox[2.3em][l]{%
				\ignorespaces#1 #2.\hfill}}}#3}
\renewcommand{\tocsubsection}[3]{%
	\indentlabel{\@ifnotempty{#2}{\hspace*{2.3em}\makebox[2.3em][l]{%
				\ignorespaces#1 #2.\hfill}}}#3}
\renewcommand{\tocsubsubsection}[3]{%
	\indentlabel{\@ifnotempty{#2}{\hspace*{4.6em}\makebox[3em][l]{%
				\ignorespaces#1 #2.\hfill}}}#3}
\makeatother

\makeatletter
\@namedef{subjclassname@2020}{%
	\textup{2020} Mathematics Subject Classification}
\makeatother

\begin{document}
%\hfill \today \\

\title{Rigidity and Flexibility of Polynomial Entropy}

%%%
%%% ~ pre {amsart} - documentclass

\author{Samuel Roth}

\address{Mathematical Institute, Silesian University in Opava, Na Rybn\'\i\v{c}ku 1\\
74601 Opava, Czech Republic\\
samuel.roth@math.slu.cz}

\author{Zuzana Roth}

\address{Mathematical Institute, Silesian University in Opava, Na Rybn\'\i\v{c}ku 1\\
74601 Opava, Czech Republic\\
zuzana.roth@math.slu.cz}

\author{\v{L}ubom\'\i r Snoha}

\address{Faculty of Natural Sciences, Matej Bel University, Tajovsk\'eho 40\\
97401 Bansk\'a Bystrica, Slovakia\\
lubomir.snoha@umb.sk}

%%% 
%%% ~ pre {article} - documentclass

%%%%%%%%%%%%%%%%%%%%
%%%====         Abstract      =======
%%%%%%%%%%%%%%%%%%%%



\maketitle

{\centering\footnotesize To Emily. All good things start small.\par}

\begin{abstract}
We introduce the notion of a one-way horseshoe and show that the polynomial entropy of an interval map is given by one-way horseshoes of iterates of the map, obtaining in such a way an analogue of Misiurewicz's theorem on topological entropy and standard `two-way' horseshoes. Moreover, if the map is of Sharkovskii type 1 then its polynomial entropy can also be computed by what we call chains of essential intervals. As a consequence we get a rigidity result that if the polynomial entropy of an interval map is finite, then it is an integer. We also describe the possible values of polynomial entropy of maps of all Sharkovskii types. As an application we compute the polynomial entropy of all maps in the logistic family. On the other hand, we show that in the class of all continua the polynomial entropy of continuous maps is very flexible. For every value $\alpha \in [0,\infty]$ there is a homeomorphism on a continuum with polynomial entropy $\alpha$. We discuss also possible values of the polynomial entropy of continuous maps on dendrites.
\end{abstract}

\tableofcontents

%%%%%%%%%%%%%%%%%%%%%%%%%%%%%%%%%%%%%%%%%%%%%%%%%%%%%%%%%%%%%%%%%%%%%%%%%%%%%%%%%%%%%

\section{Introduction}\label{S:intro}

In the present paper we are interested in flexibility and rigidity aspects of polynomial entropy. \emph{Flexibility} in general means that for a given class of dynamical systems a considered dynamical invariant (for instance some kind of entropy) can take arbitrary values, subject only to natural restrictions. Flexibility as a program in dynamics was recently formulated by A. Katok~\cite{EK19}. For some recent flexibility results see also for instance~\cite{BKRH} and \cite{AMP}; some older results in this direction, however not using the terminology of flexibility, will be mentioned below. \emph{Rigidity} in the present paper means that a considered dynamical invariant can take only very restricted values for a given class of systems.\footnote{In~\cite[Subsection 1.7.6]{BKRH}, the phenomenon of rigidity is understood in a slightly different meaning.} One may compare this with the notion of a \emph{rigid space} -- in topology a space is called rigid if it admits only trivial continuous selfmaps, i.e.\ the identity and the constant maps.  

By a \emph{(topological) dynamical system} we mean a pair $(X,f)$, where $X$ is a nonempty compact metric space and
$f\colon X\to X$ is a continuous map. A \emph{(metric) continuum} is a nonempty compact connected metric space.

As an analogue of measure theoretical entropy, Adler, Konheim and McAndrew \cite{AKM65} introduced in any topological dynamical system $(X,f)$ the concept of \emph{topological entropy} $\htop (f)$. Recall that, due to Bowen~\cite{Bow71} and Dinaburg~\cite{Din70, Din71},
\[
\htop (f) = \lim_{\varepsilon \to 0} \limsup_{n\to \infty} \frac{\log \sep (n, \varepsilon, f)}{n}
\] 
where $\sep (n, \varepsilon, f)$ is  the  maximal cardinality of subsets of $X$ which are $(n,\varepsilon)$-separated for $f$. It is well known that topological entropy is flexible even on the interval, since a continuous interval map may have any topological entropy in $[0,\infty]$, see~\cite{ALM00} or~\cite{BC99}. Other flexibility results are that a subshift over $k$ symbols may have any topological entropy in $[0, \log k]$, see~\cite{G73}, and that for every $\alpha \in [0, \infty)$ there exists a Toeplitz subshift over finitely many symbols with topological entropy $\alpha$, see~\cite{Wil84} or \cite[Theorem 4.77]{Kur03}. The general question whether for every set $\{0\} \subseteq A \subseteq [0,\infty]$ which is closed with respect to multiples (we need this condition because $\htop(f^n) = n\htop (f)$) there exists a compact metric space $X$ such that the set $\{\htop(f) \colon f \text{ is a continuous map } X\to X\}$ coincides with the set $A$, has not been answered yet, see~\cite[Question 9.5]{SYZ20}.  Let us also mention that the obvious fact that all continuous selfmaps of rigid spaces have zero topological entropy can be viewed as a rigidity result.

Systems with zero topological entropy may still exhibit complicated behaviors. Therefore there has been a search for topological invariants that measure the complexity of dynamical systems in the zero entropy regime. Here we mention two ways how to distinguish between zero entropy systems --- topological sequence entropy and polynomial entropy (there are other ways, say one can consider various other kinds of \emph{`slow entropies'} in a broad sense, see \cite{KatTho}, \cite{HK02}, \cite{GJ16}, \cite{FGJ16} and references therein, or one can consider the \emph{entropy dimension}~\cite{deC97} or \emph{topological complexity} of the system~\cite{BHM00}; for completeness let us also mention that to distinguish between systems with infinite topological entropy one can use the notion of \emph{mean topological dimension}~\cite{LW00}). 

Thus, to distinguish between systems with zero topological entropy one can for instance use the concept of \emph{topological sequence entropy}, i.e.\ topological entropy with respect to a strictly increasing subsequence of the sequence $0,1,2, \dots$, see \cite{Good74}; there is also a survey~\cite{Can08}. Rather than looking at the topological sequence entropy with respect to a given sequence, people are more interested in the supremum of the topological sequence entropies of $(X,f)$ taken over all strictly increasing subsequences of $0,1,2,\dots$. This quantity is sometimes called the \emph{supremum topological sequence entropy}; denote it $h^*_{\rm top}(f)$. If $\htop (f)>0$ then $h^*_{\rm top}(f)=\infty$, and hence supremum topological sequence entropy is useful only for systems with zero topological entropy. It is interesting that $h^*_{\rm top}(f)$ only takes values of the form $\log n$, possibly $\log 1=0$ or $\log \infty = \infty$ \cite{HY09}. In some spaces more can be said. For instance on the interval just three values $0, \log 2$ and $\infty$ can be obtained in this way \cite{Can04}; see~\cite{SYZ20} for what is known in some other spaces. These rigidity results are accompanied by a flexibility result from~\cite{SYZ20} saying that for every set $\{0\} \subseteq A \subseteq \{0, \log 2, \log 3, \ldots\}\cup \{\infty\}$ there exists a one-dimensional continuum $X$ such that the set $\{h^*_{\rm top}(f) \colon f \text{ is a continuous map } X\to X\}$ coincides with the set $A$. 

Another way to distinguish between systems with zero topological entropy is to use the concept of \emph{polynomial (topological) entropy}. Marco introduced this invariant in~\cite{Mar09} and coined the term polynomial entropy in \cite{Mar13}\footnote{A measure-theoretic version appeared even earlier as one of the invariants in~\cite{KatTho} under the name slow entropy. Topological versions for subshifts have also been considered, eg.~\cite{Cas97}, \cite{Kur03}.}. By definition, the polynomial entropy of $(X,f)$ is
\[
\hpol (f) = \lim_{\varepsilon \to 0} \limsup_{n\to \infty} \frac{\log \sep (n, \varepsilon, f)}{\log n}~.
\] 
Obviously, $\hpol (f) \geq \htop (f)$. Similarly as with $h^*_{\rm top}(f)$, the invariant $\hpol(f)$ is also infinite whenever $\htop(f)$ is positive. Hence also polynomial entropy is useful only for systems with zero topological entropy.

In the present paper we do not add more dynamical assumptions on maps, such as expansivity, equicontinuity, distality, or any kind of smoothness and the like (for some results on polynomial entropy under additional assumptions on the map see \cite{L13}, \cite{ACM17}, \cite{Mar13}, and a recent paper \cite{CPR21} where a question from~\cite{ACM17} is answered negatively). 

The following facts are known on the possible values of polynomial entropy, in compact metric spaces or in some special kinds of compact metric spaces, for general continuous maps/homeomorphisms.

\begin{itemize}
	\item If topological entropy is zero, polynomial entropy still may be positive, even infinite~\cite[Example 3.25]{Kur03}. 
	\item Polynomial entropy on the Cantor set is very flexible; for instance for every $\alpha\in [1,\infty]$ there is a Toeplitz subshift over finitely many symbols whose polynomial entropy is $\alpha$~\cite[Proposition 4.79]{Kur03}.
	\item There is a homeomorphism on an appropriate countable union of pairwise disjoint circles with non-integer polynomial entropy~\cite[Proposition 3.5]{ACM17}.
	\item The set of values of the polynomial entropies of homeomorphisms of compact metric spaces is dense in  $(0,\infty)$~\cite[Remark 3.6]{ACM17}.
\end{itemize}

Moreover, in \cite{ACM17} the authors asked the following questions.

\begin{itemize}
	\item (Problem 3 in \cite{ACM17}) Is every positive real number the polynomial entropy of a
	homeomorphism of a compact metric space?
	\item (Problem 4 in \cite{ACM17}) If $f$ is a homeomorphism of a connected compact metric
	space with $\hpol(f)$ finite, is it true that $\hpol(f)$ is an integer number?
\end{itemize} 

In\cite{HL19} the authors, apparently not aware of~\cite{ACM17}, proved the following.

\begin{itemize}
	\item For Brouwer homeomorphisms (regarded as orientation-preserving homeomorphisms on $\mathbb{S}^2$ with a unique fixed point at $\infty$) the polynomial entropy takes arbitrary values in $\{1\}\cup[2,\infty]$,~\cite[Th\'eor\`eme 1.1]{HL19}.
\end{itemize} 

\noindent This already gives a negative answer to the above-mentioned Problem 4 and a \emph{partial} positive answer to Problem 3.  

\medskip

In spite of all the above results, some challenging questions related to flexibility and rigidity aspects of polynomial entropy still remained open. Does Problem~3 have a \emph{complete} positive answer, i.e.\ do there exist also homeomorphisms on compact metric spaces with any polynomial entropy in $[0,1) \cup (1,2)$? Furthermore, once we know that the sphere $\mathbb S^2$ is very flexible with respect to polynomial entropy of homeomorphisms (in particular, it may take non-integer values), one can ask whether there is another continuum which is important in dynamics and is, contrary to the sphere, rigid with respect to polynomial entropy. If we have only homeomorphisms in mind, then it is not difficult to show that the interval admits only homeomorphisms with polynomial entropy 0 or 1, see Corollary~\ref{C:monot}. However, on the interval only noninvertible dynamics is really interesting. Therefore the question is whether the interval exhibits some kind of rigidity with respect to polynomial entropy of \emph{continuous maps}. Is it true that the polynomial entropy of a continuous interval map is always an integer, provided it is finite? This would mean that Problem 4 has a positive answer in the case of the interval. However, even if we know the answer, a still more interesting question is how one can compute the polynomial entropy of interval maps in a more effective way than to use the definition. In the present paper we answer these questions. 

\medskip

{\bf On the rigidity side}, we show that for continuous interval maps the polynomial entropy is integer or infinite. What is more important, this result is a consequence of the theory we develop. 

First, for continuous selfmaps of a compact metric space we introduce the notion of a one-way horseshoe, see Definition~\ref{D:one-way}, and we show that a one-way $\ell$-horseshoe for $f$ (or a positive iterate of $f$) implies that $\hpol(f)\geq \ell$, see Theorem~\ref{T:hor-ent}. 

Using this result, we show in Theorem~\ref{T:type1} that the polynomial entropy of interval maps $f$ of Sharkovskii type 1 is given by one-way horseshoes for iterates of $f$ or, alternatively, by so called chains of essential intervals for $f$, introduced in Definition~\ref{D:chains}. 

Finally, in our main rigidity theorem, see Theorem~\ref{T:rigid}, we prove that even for general continuous interval maps the \emph{polynomial} entropy is given by \emph{one-way horseshoes}. This is an analogue of Misiurewicz's theorem saying that for continuous maps of an interval the \emph{topological} entropy is the result of the existence of \emph{horseshoes}, see \cite{Mis79} or \cite[Theorem 4.3.5 and historical remarks on pp. 218 -- 219]{ALM00}.

As a consequence we get that the polynomial entropy of an interval map is zero if and only if its set of periodic points is connected, see Theorem~\ref{T:hpol0}. Finally, our rigidity theorem enables us to describe possible values
of interval maps depending on the Sharkovskii type of the map, see Theorem~\ref{T:values}. In particular, the polynomial entropy of an interval map is integer or infinite (and in the case of homeomorphisms it is 0 or 1, see Corollary~\ref{C:monot}).

As an application of our theory, we compute the polynomial entropy of all maps in the logistic family, see Theorem~\ref{T:logistic}.

\medskip

{\bf On the flexibility side}, in Theorem~\ref{T:flex-C} we show that for homeomorphisms on continua, polynomial entropy can take arbitrary values in $[0,\infty]$. This already gives a complete positive answer to the above-mentioned Problem 3 from \cite{ACM17} and, similarly as in~\cite{HL19}, a strong negative answer to Problem 4. 

Then we turn our attention to a simpler class of continua, dendrites, and show that for continuous dendrite maps the polynomial entropy can take many non-integer values, see Theorem~\ref{T:flex-D} and Remark~\ref{R:dend}.

%%%%%%%%%%%%%%%%%%%%%%%%%%%%%%%%%%%%%%%%%%%%%%%%%%%%%%%%%%%%%%%%%%%%%%%%%%%%%%%%%%%%%

%%%%%%%%%%%%%%%%%%%%%%%%%%%%%%%%%%%%%%%%%%%%%%%%%%%%%%%%%%%%%%%%%%%%%%%%%%%%%%%%%%%%%

\section{Terminology and Background Results}\label{S:term}
  \addtocontents{toc}{\protect\setcounter{tocdepth}{1}}%
  \setcounter{tocdepth}{1}%

%%%%%%%%%%%%%%%%%%%%%%%%%%%%%%%%%%%%%%%%%%%%%%%%%%%%%%%%%%%%%%%%%%%%%%%%%%%%%%%%%%%%%

Let $(X,f)$ be a dynamical system, meaning that $X$ is a nonempty compact metric space with metric $d$ and $f\colon X\to X$ is a continuous map. Sometimes we write $fx$ rather than $f(x)$. By $\mathbb N_0$ we denote the set of nonnegative integers.

\subsection{Polynomial entropy and its basic properties}\label{SS:prop}

Recall that a set $S\subseteq X$ is $(n,\varepsilon,f)$-separated if for any two distinct points $x, y \in S$, $\max_{i=0,\dots,n-1} d(f^i(x), f^i(y)) > \varepsilon$. If $A\subseteq X$ then a set $S\subseteq X$ (not necessarily $S\subseteq A$) is $(n,\varepsilon,f)$-spanning for the set $A$, or $(n,\varepsilon,f)$-spans the set $A$, if for every point $y\in A$ there is a point $x\in S$ such that $\max_{i=0,\dots,n-1} d(f^i(x), f^i(y))\leq \varepsilon$. Let $\sep (n, \varepsilon, f, A)$ be the maximal cardinality of subsets of $A$ which are $(n,\varepsilon,f)$-separated and let $\spanning (n, \varepsilon, f, A)$ be the minimal cardinality of subsets of $X$ which $(n,\varepsilon,f)$-span the set $A$. These quantities are finite because $X$ is compact. Then 
\[
\hpol(f,A):= \lim_{\varepsilon \to 0} \hpol^{\varepsilon}(f,A) 
\] 
where 
\[
\hpol^{\varepsilon}(f,A):= \limsup_{n\to \infty} \frac{\log \sep (n, \varepsilon, f, A)}{\log n} = \limsup_{n\to \infty} \frac{\log \spanning (n, \varepsilon, f, A)}{\log n}~.
\] 
Note that if $0< \varepsilon_1 < \varepsilon_2$ then $\hpol^{\varepsilon_1}(f,A) \geq  \hpol^{\varepsilon_2}(f,A)$, so the limit in the definition of $\hpol(f,A)$ does exist. If $A_1 \subseteq A_2$, then $\hpol(f,A_1) \leq \hpol(f,A_2)$. If $A=X$, we usually write just $\sep (n, \varepsilon, f)$, $\spanning (n, \varepsilon, f)$,  $\hpol^{\varepsilon}(f)$ and finally $\hpol(f)$, the \emph{polynomial entropy} of~$f$. If no misunderstanding can arise, we sometimes write $(n,\varepsilon)$ instead of $(n, \varepsilon, f)$.

We list some properties of the polynomial entropy, see~\cite{Mar13}: 
\begin{enumerate}
	\item [(1)] {\bf Invariance.} $\hpol(f)$ is an invariant of topological conjugacy and does not depend on the choice of topologically equivalent metrics on $X$.
	\item [(2)] {\bf Factors.} If $(Y,g)$ is a factor of $(X,f)$, i.e.\ there is a continuous surjection $\pi\colon X\to Y$ with $\pi \circ f = g \circ \pi$, then $\hpol (g) \leq \hpol (f)$.
	\item [(3)] {\bf Product rule.} If $(X,f)$ and $(Y,g)$ are dynamical systems, then $\hpol (f \times g) = \hpol (f) + \hpol (g)$.
	\item [(4)] {\bf Power rule.} For $n\geq 1$, $\hpol (f^n) = \hpol (f)$ (different from topological entropy!). If $f$ is a homeomorphism, then also $\hpol (f^{-1})=\hpol (f)$. 
	\item [(5)] {\bf Union property.} If $A = \bigcup _{i=1}^m A_i$ then $\hpol (f, A) = \max_{1\leq i \leq m} \hpol (f, A_i)$.
\end{enumerate}

\smallskip

The following lemma is a slight generalization of \cite[Proposition 2.1]{ACM17}.

\begin{lem}\label{L:Artigue}
Let $(X,d)$ be a compact metric space and $f\colon X\to X$ a continuous map. If $A\subseteq X$ contains a non-recurrent point and $f(A)\supseteq A$, then $\hpol(f,A)\geq 1$.	
\end{lem}

\begin{proof}
Choose a non-recurrent point $x_0\in A$ and $\varepsilon >0$ such that $d(f^n(x_0), x_0) >\varepsilon$ for all $n\geq 1$. Since $A$ $f$-covers itself, for every $n$ we can choose points
$x_{-1}, \dots, x_{-n+1} \in A$ such that $f(x_{-i})=x_{-i+1}$, $i=1, \dots, n-1$. We claim that $\{x_0,\dots, x_{-n+1}\}$ is an $(n,\varepsilon)$-separated set. Indeed, if $0\leq i < j \leq n-1$, then $d(f^j(x_{-i}), f^j(x_{-j})) = d(f^{j-i}(x_0), x_0)>\varepsilon$. This shows that 
$\sep (n,\varepsilon, f, A) \geq n$. Then
\[
\hpol(f,A) \geq \hpol^{\varepsilon}(f,A) = \limsup_{n\to \infty} \frac{\log \sep (n, \varepsilon, f, A)}{\log n} \geq 1.\qedhere
\]
\end{proof}

\subsection{Polynomial vs.\ topological entropy}

The name ``polynomial entropy'' was chosen by Marco in~\cite{Mar13} because $\hpol$ measures the growth of distinguishable orbit segments at a ``polynomial level.'' This can be illustrated by the following comparison in Big-$\mathcal{O}$ notation. Topological entropy $\htop(f)$ may be characterized as the infimum of those $a>0$ such that for all $\epsilon>0$,
\begin{equation*}
\sep(n,\epsilon,f)=\mathcal{O}(e^{an}) \text{ as }n\to\infty
\end{equation*}
(so it is an exponential growth rate of complexity). In the same way, $\hpol(f)$ may be characterized as the infimum of those $a>0$ such that for all $\epsilon>0$,
\begin{equation*}
\sep(n,\epsilon,f)=\mathcal{O}(n^a) \text{ as }n\to\infty
\end{equation*}
(so it is a polynomial growth rate of complexity, except that the ``degree'' $a$ need not be an integer). Several times throughout the paper when we need to calculate polynomial entropy, we will use the fact that if $p(n)$ is a polynomial in $n$ of degree $d$, then $\frac{\log p(n)}{\log n} \to d$ as $n\to\infty$.

We have seen that polynomial entropy shares some properties with topological entropy, but they differ in the power formula. There are also other differences and we mention the following one, since we will use the described construction.

Bowen's formula~\cite[Theorem 17]{Bow71} which estimates from above the topological entropy of a skew product does not have analogue for polynomial entropy. For instance, in \cite[Proposition 3.5]{ACM17} there is a homeomorphism $f$ in the space $M= \mathbb S^1 \times (\{0\} \cup \{a_1, a_2, \dots\})$ where $(a_n)_{n=1}^\infty$ is a decreasing sequence of positive real numbers such that $a_n\to 0$. The homeomorphism $f$ acts as a rotation by angle $a_n$ on $\mathbb S^1 \times \{a_n\}$ and is the identity on $\mathbb S^1 \times \{0\}$. It is a skew product over the identity and each fibre map, being a circle rotation, also has zero polynomial entropy. However, it is shown in \cite[Proposition 3.5]{ACM17} that the numbers $a_n$ can be chosen in such a way that $\hpol(f)>0$. In particular, we see that Bowen's formula fails in this setting, and that the union property mentioned above does not extend to countable unions.

%%%%%%%%%%%%%%%%%%%%%%%%%%%%%%%%%%%%%%%%%%%%%%%%%%%%%%%%%%%%%%%%%%%%%%%%%%%%%%%%%%%%%

\subsection{Polynomial entropy of subshifts}

%%%%%%%%%%%%%%%%%%%%%%%%%%%%%%%%%%%%%%%%%%%%%%%%%%%%%%%%%%%%%%%%%%%%%%%%%%%%%%%%%%%%%

It is possible to calculate the polynomial entropy of a subshift $X \subseteq \{ 0,1\}^{\Nzero}$ by counting words. Let $ \omega(n)$ denote the number of distinct words of length $n$ in $X$. The function $\omega$  is called the {\it complexity function} of the subshift $X$.
Just as in the case of topological entropy (where, however, the limit exists since the denominator is $n$ rather than $\log n$), we have the following.

\begin{lem}\label{LemWords}
The polynomial entropy of a subshift $X$ is given by
$${\hpol}\left( \sigma |_X \right) = \limsup_{n \to \infty} \dfrac{\log \omega(n)}{\log n}.$$
\end{lem}

In \cite[Th\'eor\`{e}me 6.1]{Cas97} Cassaigne proved that for every $\alpha\in(1,2)$ there is a 1-sided infinite sequence $\underline{u}\in\{0,1\}^{\mathbb{N}_0}$ such that the complexity function $p_{\underline{u}}(n)$ which counts the number of distinct length $n$ words occuring in $\underline{u}$ is asymptotic with $n^\alpha$, meaning $\frac{p_{\underline{u}}(n)}{n^\alpha} \to 1$ as $n\to\infty$. Notice that in this case
\begin{equation*}
\limsup_{n\to\infty} \frac{\log p_{\underline{u}}(n)}{\log n} = \limsup_{n\to\infty} \frac{\log(n^\alpha)}{\log n} = \alpha.
\end{equation*}
Let $X$ denote the subshift obtained by taking the orbit closure of $\underline{u}$ under the shift map. On one hand, one can see from the construction in \cite{Cas97} that $\underline{u}$ is recurrent (though not syndetically recurrent) and so the orbit closure $X$ is surjective (though not minimal). On the other hand, it is well known that the number $\omega(n)$ of words of length $n$ in the orbit closure $X$ is the same as $p_{\underline{u}}(n)$. Therefore by Lemma~\ref{LemWords} we have the following.

\begin{prop}[Cassaigne, \cite{Cas97}]\label{P:Cass}
For every $\alpha\in(1,2)$ there is a surjective subshift $X$ in 2 symbols whose polynomial entropy is $\alpha$.
\end{prop}

There is a stronger flexibility result.

\begin{prop}[{K\r{u}rka, \cite[Proposition 4.79]{Kur03}}]\label{P:Kurka}
For every $\alpha\in[1,\infty]$ there is a Toeplitz subshift $X$ in finitely many symbols whose polynomial entropy is $\alpha$.
\end{prop}

K\r{u}rka's result is especially nice since it provides examples of minimal dynamical systems with flexible polynomial entropy, much in the spirit of Grillenberger's construction of uniquely ergodic shifts with flexible topological entropy~\cite{G73}. However, we will give preference to Cassaigne's subshifts in our work, since they are constructed in a binary alphabet.

%%%%%%%%%%%%%%%%%%%%%%%%%%%%%%%%%%%%%%%%%%%%%%%%%%%%%%%%%%%%%%%%%%%%%%%%%%%%%%%%%%%%%

\section{One-way Horseshoes}\label{S:one-way}
  \addtocontents{toc}{\protect\setcounter{tocdepth}{2}}%
  \setcounter{tocdepth}{2}%

%%%%%%%%%%%%%%%%%%%%%%%%%%%%%%%%%%%%%%%%%%%%%%%%%%%%%%%%%%%%%%%%%%%%%%%%%%%%%%%%%%%%%

We introduce the notion of a one-way horseshoe which will play the same role in the theory of \emph{polynomial} entropy of interval maps as the notion of the horseshoe plays in the theory of \emph{topological} entropy of interval maps. By Misiurewicz's theorem \cite[Theorem 4.3.5]{ALM00}, the topological entropy of an interval map $f$ is given by
horseshoes of iterates of $f$, and similarly it will turn out that the polynomial entropy of an interval map is given by one-way horseshoes of iterates of $f$, see Theorem~\ref{T:rigid}. Moreover, in both cases it is sufficient to look for horseshoes formed by closed intervals.

In this section, we use one-way horseshoes to give a lower bound for polynomial entropy that applies to any topological dynamical system, see Theorem~\ref{T:hor-ent}.

\begin{definition}\label{D:one-way}
Let $f:X\to X$ be a continuous map on a compact metric space. A \emph{one-way horseshoe of order $\ell$} (we will sometimes call it a \emph{one-way $\ell$-horseshoe}) is an indexed family $A_1, \ldots, A_{\ell}$ of pairwise disjoint compact sets such that $A_\ell$ (the last set in the family) contains a non-recurrent point of $f$ and $f(A_i)\supseteq A_j$ for all $1\leq i \leq j \leq \ell$.
\end{definition}

\begin{figure}[h!]
	\begin{tikzpicture}[scale=0.7]
	\begin{scope}
	\clip (0,0) rectangle (7,7);
	\draw [very thin] (0,0) rectangle (7,7);
	\draw [very thin] (0,0) -- (7,7);
	\draw [very thin] (0,0) grid (7,7);
	\draw [thick] (0,0) -- (1/2,7) -- (1,1) --(2,2) -- (5/2,7) -- (3,3) -- (4,4);
	\draw [thick] plot [smooth] coordinates {(4,4) (4.3, 4.6) (5,5)}; 
	\draw [thick] (5,5) -- (6,6) -- (6.5,4) -- (7,7);
	\end{scope}
	\node [above] at (3.5,7) {$f$};
	\node [below] at (0.5,0) {$A_1$};
	\node [below] at (2.5,0) {$A_2$};
	\node [below] at (6.5,0) {$A_3$};
	\node [below] at (4.5,0) {$A_4$};
	\end{tikzpicture}
	\caption{The compact sets (in fact intervals) $A_1, A_2, A_3, A_4$ form a one-way $4$-horseshoe for $f$. The set $A_4$ clearly contains a non-recurrent point.}\label{F:horse}
\end{figure}
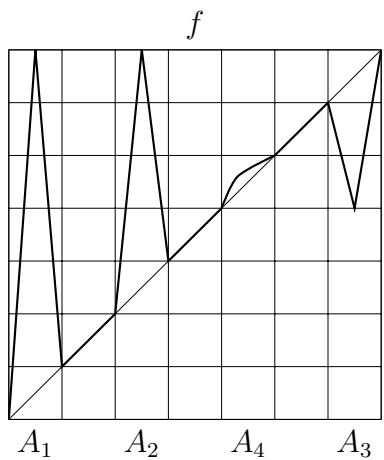

\begin{rem}\label{R:nonrec}
In the definition we require the existence of a non-recurrent point only in the last set $A_{\ell}$, see also Fig.~\ref{F:horse}. The covering properties of the sets in the one-way horseshoe imply that, if $\ell \geq 2$, each of the previous sets $A_i$ contains a non-recurrent point automatically, even without assuming that $A_{\ell}$ contains such a point (if $i< \ell$, choose a point $x\in A_i$  which travels into $A_{\ell}$ and then remains there for all time; this is possible because, due to compactness of $A_{\ell}$, there is a point in $A_{\ell}$ whose forward orbit is a subset of $A_{\ell}$). This in particular implies that if $A_1, \ldots, A_{\ell}$ is a one-way $\ell$-horseshoe, then any nonempty subfamily $A_{i_1}, \ldots, A_{i_k}$ with $1\leq i_1 < i_2 < \ldots < i_k \leq \ell$ is a one-way $k$-horseshoe. 

Observe also that in the definition we really need to assume that the last set $A_{\ell}$ contains a non-recurrent point.  Otherwise the existence of a one-way $\ell$-horseshoe would not imply that $\hpol(f)\geq \ell$ (for $\ell =1$ consider the identity), though by Theorem~\ref{T:hor-ent} we could claim that $\hpol(f)\geq \ell -1$ because $A_1, \ldots, A_{\ell-1}$ would be a one-way $\ell -1$ horseshoe (recall that $A_{\ell -1}$ contains a non-recurrent point automatically).

Moreover, it is worth recalling that each iterate $f^n$ has the same set of recurrent points as $f$ \cite[Lemma 25, p.82]{BC92}, so the presence of a non-recurrent point in $A_{\ell}$ does not depend on which iterate we use when looking for the one-way horseshoe.
\end{rem}

\begin{lem}\label{L:horse-iter}
If $A_1, \dots, A_{\ell}$ is a one-way $\ell$-horseshoe for $f$, then it is a one-way $\ell$-horseshoe for every iterate $f^n$, $n=1,2,\dots$.  	
\end{lem}

\begin{proof}
$A_{\ell}$ contains a non-recurrent point by the assumption.  If $1\leq i \leq j \leq \ell$, then $f^n(A_i)=f(f^{n-1}A_i) \supseteq f(A_i) \supseteq A_j$.
\end{proof}

Easy examples show that the converse to this lemma is not true.

\medskip

By Lemma~\ref{L:Artigue}, if $f$ has a one-way horseshoe $A$ of order $1$, then $\hpol (f)\geq 1$. We prove more.

\begin{thm}\label{T:hor-ent}
Let $(X,d)$ be a compact metric space and $f\colon X\to X$ a continuous map. If $f$ (or a positive iterate of $f$)  has a one-way $\ell$-horseshoe, then $\hpol (f) \geq \ell$.
\end{thm}

\begin{proof}
In view of the power rule for polynomial entropy we may assume that $f$ itself has a one-way horseshoe. Fix $\ell \geq 1$ and let $A_1, \dots, A_{\ell}$ form a one-way $\ell$-horseshoe for $f$. We remark that the case $\ell =1$ is covered by Lemma~\ref{L:Artigue} and so we may assume that $\ell \geq 2$. Choose a non-recurrent point $x\in A_{\ell}$ and $\varepsilon >0$ small enough that
\begin{enumerate}
	\item $d(A_i, A_j)> \varepsilon$ for $1\leq i < j \leq \ell$, and
	\item $d(f^m(x), x) > \varepsilon$ for $m\geq 1$.
\end{enumerate}	
A word of length $k\in \mathbb N_0$ is a sequence $\underline{s} = s_0\dots s_{k-1} \in \{1,\dots, \ell\}^k$, and its length is denoted by $|\underline{s}|=k$. There are exactly $\ell^k$ words of length $k$ (note that the empty word $\emptyset$ is the unique word of length $0$). Let $\NDW (k)$ denote the set of \emph{nondecreasing} words of length $k$, 
where $\underline{s}$ is nondecreasing if $s_i \leq s_j$ whenever $0\leq i < j <  |\underline{s}|$.
Because of the $f$-covering properties of the sets $A_1, \dots, A_{\ell}$ in the one-way horseshoe, for every $\underline{s} \in \NDW(k)$ there is a point $x_{\underline{s}}$ such that
$f^i(x_{\underline{s}})\in A_{s_i}$ for $i\in \{0,\dots,k-1\}$ and $f^k(x_{\underline{s}})=x$. (Notice that, if $k=0$, $x_{\emptyset}=x$.) We claim that the set
\[
E_n = \left\{x_{\underline{s}}\colon \underline{s} \in \bigcup_{k=0}^{n-1}\NDW(k)\right\}
\]
is $(n,\varepsilon)$-separated for $f$. For suppose $x_{\underline{s}}, x_{\underline{t}} \in E_n$, $\underline{s} \neq \underline{t}$. On one hand, if $\underline{s}$, $\underline{t}$ both have the same length, then there is $i$ such that $s_i\neq t_i$. Then $f^i(x_{\underline{s}}) \in A_{s_i}$, $f^i(x_{\underline{t}}) \in A_{t_i}$, so by~(1) we have
$d(f^i(x_{\underline{s}}), f^i(x_{\underline{t}}))>\varepsilon$. On the other hand, if $\underline{s}$ is shorter than $\underline{t}$, then $d(f^{|\underline{t}|}(x_{\underline{s}}), f^{|\underline{t}|}(x_{\underline{t}})) = d(f^{|\underline{t}|-|\underline{s}|}(x), x)>\varepsilon$ by (2). This shows that $E_n$ is $(n,\epsilon)$-separated, and additionally that distinct words $\underline{s}$, $\underline{t}$ lead to distinct elements $x_{\underline{s}}, x_{\underline{t}}$ of $E_n$. As a consequence we have
\[
\sep(n, \varepsilon, f) \geq  \# E_n = \sum_{k=0}^{n-1} \#\NDW(k).
\]
It remains to count the number of nondecreasing words of length $k$. For each $\underline{s} \in \NDW(k)$ let $\tau(\underline{s}) = (\tau_1, \tau_2, \dots, \tau_{\ell})$ be the ordered $\ell$-tuple where $\tau_i \geq 0$ is the number of occurrences of the symbol $i$ in the word 
$\underline{s}$; for example, if $k=7$, $\ell=4$ and $\underline{s} = 2233334$ then $\tau(\underline{s}) = (0,2,4,1)$. Clearly, $\tau$ gives a bijection of $\NDW(k)$ with the set of ordered $\ell$-tuples of nonnegative integers whose sum is $k$ (the number of such tuples is the same as the number of ways to distribute $k$ identical objects into $\ell$ distinct boxes). Thus, by the stars and bars theorem,
\[
\#\NDW(k) = \binom{k+\ell -1}{\ell -1}.
\]
Then $\# E_n = \sum_{k=0}^{n-1} \binom{k+\ell -1}{\ell -1}$ and by the hockey-stick combinatorial identity $\sum_{i=r}^m \binom{i}{r} = \binom {m+1}{r+1}$ we get 
\[
\# E_n = \binom{n+\ell -1}{\ell}.
\]
But this is a polynomial in $n$ of degree $\ell$. Therefore
\[
\hpol(f) \geq \hpol^{\varepsilon}(f) \geq \limsup_{n\to \infty} \frac{\log \# E_n}{\log n} = \limsup_{n\to \infty} \frac{\log \binom{n+\ell -1}{\ell}}{\log n} = \ell.\qedhere
\]
\end{proof}

%%%%%%%%%%%%%%%%%%%%%%%%%%%%%%%%%%%%%%%%%%%%%%%%%%%%%%%%%%%%%%%%%%%%%%%%%%%%%%%%%%%%%

\section{Rigidity of Polynomial Entropy on the Interval}
\subsection{Maps of Sharkovskii type 1}\label{SS:type1}

%%%%%%%%%%%%%%%%%%%%%%%%%%%%%%%%%%%%%%%%%%%%%%%%%%%%%%%%%%%%%%%%%%%%%%%%%%%%%%%%%%%%%

We now focus on maps of the interval, where one-way horseshoes provide not only a lower bound, but in fact determine the polynomial entropy. We start with maps $f:[0,1]\to[0,1]$ of \emph{Sharkovskii type 1}, defined by the condition that all periodic points of $f$ are in fact fixed points. This simplifies the analysis considerably, and we return to maps of other Sharkovskii types later. 

\begin{lem}[\cite{Cop55}, \cite{Sh65}]\label{lem:sharkovskii}
	Let $f$ have Sharkovskii type 1. Then
	\begin{enumerate}
		\item \emph{Same-side rule:} All images of a non-fixed point $x$ lie on the same side of $x$ as the first image. That is, if $f(x)>x$, then $f^n(x)>x$ for all $n\geq1$, and if $f(x)<x$ then $f^n(x)<x$ for all $n\geq 1$.
		\item \emph{End behavior:} Every trajectory of $f$ converges to a fixed point of $f$.
		\item \emph{Nonwandering points:} Every non-wandering point of $f$ is fixed.
	\end{enumerate}
\end{lem}

Part (3) of Lemma~\ref{lem:sharkovskii} is not explicitly stated in~\cite{Cop55},~\cite{Sh65}, but it follows easily from the other two parts. We include the proof for completeness.

\begin{proof}[Proof of Lemma~\ref{lem:sharkovskii}~(3)]
Suppose $x$ is non-wandering but not fixed. We may assume that $x<f(x)$. By the same side rule also $x<f^2(x)$. By continuity we can choose a small connected neighborhood $U$ of $x$ which is to the left of both $U_1=f(U)$ and $U_2=f^2(U)$ (although $U_1$ and $U_2$ may overlap). Since $x$ is non-wandering there is a point $y$ in $U$ whose trajectory returns to $U$. Write $y_i=f^i(y)$ for each natural number $i$ and let $m$ be such that $y_m \in U$. Clearly $m>2$. Since $y_m \in U$ and $y_1\in U_1$ we have $y_m <y_1$ and so by the same-side rule $y_i < y_1$ for all $i> 1$ (so $y_1$ is the right-most point in the forward trajectory of $y$). Let $j$ be the smallest integer in the range $2\leq j\leq m$ so that $y_j$ is not in $U_1$. Since $U_2 = f(U_1)$ the minimality of $j$ guarantees that $y_j \in U_2$, so $y_j$ is to the right of $U$ (hence $j<m$). On the other hand, the connected set $U_1$ contains $y_1$ and does not contain $y_j<y_1$. So $y_j$ lies between $U$ and $U_1$. But then $y_m\in U$ lies to the left of $y_j$, and $y_{m+1}\in U_1$ lies to the right of $y_j$, contradicting the same-side rule.
\end{proof}

By $\Fix(f)$ or $\Per(f)$ we denote the set of fixed points or periodic points of $f$, respectively. If $J\subseteq [0,1]$ then the \emph{orbit} of $J$ is the set $\Orb(J)= \bigcup_{n=0}^{\infty} f^n(J)$.

As a technical tool which will allow us to construct one-way horseshoes, we introduce the following notion.

\begin{definition}\label{D:chains}
Let $f$ have Sharkovskii type 1. An interval $I=(a,b)$ with $a,b\in\Fix(f)$ and $I\cap\Fix(f)=\emptyset$ is called an \emph{essential interval}. An indexed family $I_1, I_2, \ldots, I_{\ell}$ of distinct essential intervals such that $\Orb(I_i) \supseteq I_{i+1}$ for all $i<\ell$ is called a \emph{chain of essential intervals of length $\ell$} (or more simply an \emph{$\ell$-chain of essential intervals}).
\end{definition}

We are coming to our main rigidity theorem for interval maps of Sharkovskii type 1. From now on we take zero to be the supremum of the empty set, so for example if $f$ has no essential intervals, then the quantity in part (b) of the next theorem is zero.

\begin{thm}\label{T:type1}
For an interval map $f$ of Sharkovskii type 1, the following four quantities are equal (and they belong to $\mathbb N_0 \cup \{\infty\}$).
\begin{enumerate}
	\item [(a)] The polynomial entropy $\hpol (f)$.
	\item [(b)] The supremum of the lengths of chains of essential intervals for $f$.
	\item [(c)] The supremum of the orders of one-way horseshoes of $f$ and its iterates.
	\item [(d)] The supremum of the orders of one-way horseshoes of $f$ and its iterates consisting of (closed) intervals.
\end{enumerate}
Consequently, if $\hpol(f)$ is finite, then it is a nonnegative integer.
\end{thm}

\begin{proof}[Proof outline]
To prove Theorem~\ref{T:type1} we will establish that for any given integer $\ell\geq 1$, the following statements are equivalent.
\begin{enumerate}[(i)]
\item $f$ has an $\ell$-chain of essential intervals $I_1, \ldots, I_\ell$.
\item Some iterate of $f$ has a one-way $\ell$-horseshoe $A_1, \ldots, A_\ell$ with each $A_i$ an interval.
\item Some iterate of $f$ has a one-way $\ell$-horseshoe $A_1, \ldots, A_\ell$.
\item $\hpol(f)\geq\ell$.
\item $\hpol(f)>\ell-1$.
\end{enumerate}
(Notice that without (v) in this list of equivalent conditions it would not be apriori excluded that  $\hpol(f)$ is finite and non-integer.)

We have already shown (iii)$\implies$(iv) in Theorem~\ref{T:hor-ent}. The implications (iv)$\implies$(v) and (ii)$\implies$(iii) are trivial. 
To prove (i)$\implies$(ii) we show how to construct a one-way horseshoe from a chain of essential intervals in Proposition~\ref{P:chain-hor} below. The implication (v)$\implies$(i) is established in Proposition~\ref{P:ell-chain} through a coding argument which takes up all of Subsection~\ref{SS:coding}.
\end{proof}

Before finishing the proof of Theorem~\ref{T:type1}, we mention a few corollaries. First, we are able to deduce when a type 1 map has zero polynomial entropy.

\begin{cor}\label{C:zero}
An interval map of Sharkovskii type 1 has zero polynomial entropy if and only if its set of fixed points is connected.
\end{cor}

\begin{proof}
If the fixed point set is connected, then there are no essential intervals, so there can be no chains of essential intervals and the supremum of the lengths of those chains is zero. Then the polynomial entropy is also zero. But if the fixed point set is not connected, then there is at least one essential interval, which is already a chain of length one, and so the polynomial entropy is at least one.
\end{proof}

Another useful corollary is about (not necessarily strictly) monotone maps. 

\begin{cor}\label{C:monot}
A monotone interval map, in particular a homeomorphism, has polynomial entropy either 0 or 1, depending on whether the set of periodic points is connected or not, respectively.  
\end{cor}

\begin{proof}
For such a map $f$, $g=f^2$ is (not necessarily strictly) increasing and so is of Sharkovskii type 1. Moreover, $\Per(f)=\Fix(g)$. By the power rule,  $\hpol (f) =\hpol (g)$. If $\Fix(g)$ is connected, we have $\hpol(g)=0$ by Corollary~\ref{C:zero}. Otherwise $g$ has an essential interval but, since it is increasing, no nontrivial chain of essential intervals. Therefore $\hpol (g)=1$ by Theorem~\ref{T:type1}.
\end{proof}

\begin{rem}
In \cite{L13} it is proved that the polynomial entropy of an orientation preserving homeomorphism of the circle equals $1$ when the homeomorphism is not conjugate to a rotation and it is $0$ otherwise. Lemma 3.1 in the same preprint says something for interval maps which also follows from our Theorem~\ref{T:type1}: Let $I = [a, b]$ be a compact interval in $\mathbb R$ and let $f\colon I\to I$ be a continuous, increasing function such that $f(a) = a$, $f(b)=b$ and $f(x)-x\neq 0$ for all $x\in (a, b)$. Then $\hpol(f) = 1$. It was also shown in \cite{GC21} that a monotone map $f:I\to I$ on a compact interval $I=[a,b]$ has polynomial entropy either 0 or 1, depending on whether or not the second iterate $f^2$ restricted to the core $\bigcap_{n=0}^\infty f^n(I)$ is equal to the identity.
\end{rem}

Now we begin to show how a chain of essential intervals can be used to construct a one-way horseshoe.

\begin{definition}
Let $f$ have Sharkovskii type 1, and let $\I=\I(f)$ denote the set of all essential intervals. We classify each $I\in\I$ as an \emph{up} (respectively \emph{down}) interval if $f(x)>x$ (respectively $f(x)<x$) for all $x\in I$. The \emph{source} of $I=(a,b)\in \I$ is $a$ if $I$ is an up interval and $b$ if $I$ is a down interval. This defines a function $\I \ni I \mapsto \source(I) \in \Fix(f).$
\end{definition}

\begin{lem}\label{lem:orb-ess}
For an interval map $f$ of Sharkovskii type 1,
\begin{itemize}
\item the orbit $\Orb(I)$ of an essential up interval $I=(x,y)$ is of the form $(x,z)$ or $(x,z]$ with $z\geq y$,
\item the orbit $\Orb(I)$ of an essential down interval $I=(y,x)$ is of the form $(z,x)$ or $[z,x)$ with $z\leq y$
\end{itemize}
\end{lem}
\begin{proof}
Let $I=(x,y)$ be an essential up interval. Since $f(I)\supseteq I$ it follows that $\Orb(I)=\bigcup_{n=0}^\infty f^n(I)$ is the union of an increasing sequence of intervals, hence an interval. By Lemma~\ref{lem:sharkovskii} it follows that every point in $\Orb(I)$ lies to the right of $x$. This shows that $\Orb(I)$ is of the form $(x,z)$ or $(x,z]$ with $z\geq y$. The analogous statement for essential down intervals follows just as easily.
\end{proof}

\begin{prop}\label{P:chain-hor}
If an interval map $f$ of Sharkovskii type 1 has a chain of essential intervals of length $\ell\geq 1$, then some iterate of $f$ has a one-way $\ell$-horseshoe composed of pairwise disjoint closed intervals $A_1, \ldots, A_{\ell}$.
\end{prop}
\begin{proof}
Let $I_1, I_2, \ldots, I_{\ell}$ be the chain of essential intervals. For each $i\in\{1,\ldots,\ell\}$, let $I_i = (a_i, b_i)$ and let $A_i$ be the closed interval joining the source of $I_i$ to its midpoint, that is,
\begin{equation}\label{Ai}
A_i=\begin{cases}
\left[a_i, \frac{a_i+b_i}{2}\right], & \text{ if $I_i$ is an up interval,}\\
\left[\frac{a_i+b_i}{2}, b_i\right], & \text{ if $I_i$ is a down interval.}
\end{cases}
\end{equation}
We will show that $A_1, \ldots, A_{\ell}$ is a one-way horseshoe for some iterate of $f$.

First we show that $A_1, \ldots, A_{\ell}$ are pairwise disjoint. Fix $1\leq i < j \leq \ell$. We may suppose that $I_i$ is an up interval (the argument for a down interval is symmetric). Then $\source(I_i)=a_i$ and by Lemma~\ref{lem:orb-ess} the orbit of $I_i$ lies to the right of $a_i$. On the other hand, $\Orb(I_i) \supseteq I_j$ by the definition of a chain of essential intervals, so $I_j \neq I_i$ must lie to the right of $b_i$. It follows from~\eqref{Ai} that $A_i \cap A_j = \emptyset$.

By Lemma~\ref{lem:sharkovskii} every point in an essential interval of $f$ is wandering, so each $A_i$ clearly contains non-recurrent points.

It remains to show that there is an iterate $k$ such that $f^k(A_i) \supseteq A_j$ for all $1\leq i \leq j\leq \ell$. It is clear from~\eqref{Ai} that $f(A_i) \supseteq A_i$ for all $i$ (and this finishes the proof in the case $\ell=1$). Then, since there are only finitely many pairs $i<j$, it suffices to find for each such pair a number $m=m(i,j)$ such that $f^m(A_i) \supseteq A_j$, and take $k=\max_{i<j} m(i,j)$. So fix a pair $i,j$ with $1\leq i < j \leq \ell$. We will show how to find $m$. We may assume that $I_i$ is an up interval (the argument for a down interval is symmetric). Again using Lemma~\ref{lem:orb-ess} and the definition of a chain of essential intervals, it follows that $\Orb(I_i)$ is an interval of the form $(a_i,z)$ or $(a_i,z]$ lying to the right of $a_i$ and containing $I_j$. Since the left endpoint of $A_i$ is fixed and the rest of $A_i$ is in the up interval $I_i$ it is easy to see that $A_i \subseteq f(A_i)$, which implies that $\Orb(A_i)=\bigcup_{n=0}^\infty f^n(A_i)$ is the union of an increasing sequence of intervals. If the right endpoint of $A_i$ never leaves $I_i$, then its trajectory is a monotone sequence and therefore converges to a fixed point. Since there are no fixed points in $I_i$ this shows that $\Orb(A_i) \supseteq I_i$. So far we have shown that
\begin{equation}\label{OrbAi}
A_i \subseteq f(A_i) \subseteq f^2(A_i) \subseteq \cdots \ \text{ and } \ \Orb(A_i)\supseteq\Orb(I_i)\supseteq I_j.
\end{equation}

We complete the proof in two cases. First suppose that $I_j$ is also an up interval. By~\eqref{OrbAi} there is $m$ such that $f^m(A_i) \ni \frac{a_j+b_j}{2}=\max(A_j)$. Since $f^m(A_i)$ is connected it follows that $f^m(A_i)\supseteq A_j$. Now suppose instead that $I_j$ is a down interval. Again by~\eqref{OrbAi} there is $m'$ such that $f^{m'}(A_i) \ni \frac{a_j+b_j}{2}=\min(A_j)$. We have that $f^{m'}(A_i)=[a_i,z']$ is a closed interval, and if $z'\geq b_j=\max(A_j)$ then we put $m=m'$ and we are finished. Otherwise $z'\in [\min(A_j),\max(A_j))$. If $f^{m'+1}(A_i)\ni\max(A_j)$ then we choose $m=m'+1$ and again we are finished. Otherwise $f^{m'+1}(A_i)=[a_i, z'']$ with $z'\leq z'' < b_j=\max(A_j)$. But $[a_i, z''] = [a_i,z'] \cup[z',z'']$ where $[z', z'']$ is a part of the down interval $I_j$. Therefore $$f^{m'+2}(A_i) = f([a_i, z']) \cup f([z', z'']) = f^{m'+1}(A_i) \cup f([z',z'']).$$ Here $[z', z'']$ is mapped to the left since it is in a down interval, but not further left than $a_i$ since it is contained in $\Orb(I_i)$. This shows that $f^{m'+2}(A_i)=f^{m'+1}(A_i)$ and thus the nested sequence in~\eqref{OrbAi} stabilizes at time $m'+1$. Therefore $$\Orb(I_i)=\Orb(A_i)=f^{m'+1}(A_i)=[a_i,z'']\not\supseteq I_j,$$ contradicting the fact that $I_1, \ldots, I_{\ell}$ is a chain of essential intervals for $f$.
\end{proof}

\begin{rem}\label{r:no-loops}	
As a corollary to the proof we see that a pair of distinct essential intervals cannot each contain the other in its orbit. Otherwise we would get closed disjoint intervals $A_1, A_2$ drawn from source to midpoint with $f^m(A_1) \supseteq A_2$ and vice-versa, forming a standard ``two-way'' horseshoe and giving $f$ periodic points of periods other than 1 (as well as positive topological entropy), which is impossible for a map of Sharkovskii type 1.	
\end{rem}

\subsection{A coding argument}\label{SS:coding}
To finish proving Theorem~\ref{T:type1} we need to show that given any positive integer $\ell$, for any interval map $f$ of Sharkovskii type 1 with $\hpol(f)>\ell-1$ we can find an $\ell$-chain of essential intervals (this is Proposition~\ref{P:ell-chain} below). Our construction follows a coding argument: each point $x$ is assigned a code, i.e.\ a sequence chosen from a finite alphabet of points in $[0,1]$ which closely shadows its orbit. The code consists predominantly of constant blocks of the form $c\cdots c$ which repeat some fixed point $c\in\Fix(f)$ several times while the orbit of $x$ is close to that fixed point. The block stops if the orbit moves far enough away from $c$, but since, by Lemma~\ref{lem:sharkovskii}, each orbit of a type 1 map eventually converges to a fixed point, each code can be chosen to eventually enter an infinite block of the form $c c \cdots$. The switching between constant blocks allows us to construct a chain of essential intervals using three technical lemmas which, loosely speaking, state the following facts:
\begin{enumerate}
\item If an orbit is close to a fixed point $c$ and then moves far away, it never returns close to $c$ again, see Lemma~\ref{lem:never-return}.
\item If an orbit is close to a fixed point $c$ and then moves far away, it must first pass through an essential interval $I$ of $f$ at a point which is close to the ``source'' of $I$ (the definition is below), while $\source(I)$ itself is close to $c$, see Lemma~\ref{lem:find-ess-int}.
\item If the sources of two essential intervals $I, J$ are far apart and an orbit travels first through $I$ and later into $J$ at a point close to $\source(J)$, then $\Orb(I)\supseteq J$, see Lemma~\ref{lem:covering}.
\end{enumerate}
Using these three facts we can show that if a code makes $\ell$ switches from one constant block to another, then $f$ has a chain of $\ell$ distinct essential intervals: their existence follows from fact (2), their distinctness from fact (1), and the covering property from fact (3). But on the other hand, a combinatorial argument shows that some codes must switch blocks at least $\ell$ times when $\hpol(f)>\ell-1$.

We now make these ideas more precise.

\begin{lem}\label{lem:2up}
Let $f$ be a map of Sharkovskii type 1. The orbits of two essential intervals of the same type (both up or both down) are either disjoint, or else one of these orbits contains the other.
\end{lem}
\begin{proof}
Let $I,J$ be two essential intervals and suppose without loss of generality that both $I,J$ are up intervals and $I$ lies to the left of $J$. By Lemma~\ref{lem:orb-ess} $\Orb(I)$ is of the form $(x,z)$ or $(x,z]$, where $x$ is the left endpoint of $I$. Since $\Orb(I)$ is invariant, $z$ cannot lie in an up interval, so $z\not\in J$. If $z$ lies to the left of $J$, then $\Orb(I)\cap\Orb(J)=\emptyset$. But if $z$ lies to the right of $J$, then $\Orb(I)\supseteq\Orb(J)$.
\end{proof}

\begin{lem}\label{lem:never-return}
Let $f$ be a map of Sharkovskii type 1. For any $\epsilon>0$ there exists $\delta>0$ such that for $x\in[0,1]$, $c\in\Fix(f)$, $t\in\N$, if $|x-c|<\delta$ and $|f^t(x)-c|\geq\epsilon$, then $|f^m(x)-c|\geq\delta$ for all $m\geq t$.
\end{lem}
\begin{proof}
Use the uniform continuity of $f$ to choose a positive number $\delta<\epsilon$ such that points which are $\delta$-close have images which are $\epsilon$-close. Then since $c$ is fixed, the point $f^{t-1}(x)$ cannot belong to the ball $B(c,\delta)$. Without loss of generality we may suppose that $f^t(x)\geq c+\epsilon$. By Lemma~\ref{lem:sharkovskii} the points $f^{t-1}(x), f^t(x)$ both lie on the same side of $x$, so we have $x<c+\delta\leq f^{t-1}(x)<f^t(x)$. Now by Lemma~\ref{lem:sharkovskii} applied to the point $f^{t-1}(x)$, each $f^m(x)$ with $m\geq t$ lies to the right of $f^{t-1}(x)$.
\end{proof}

If $I, J\in \I$ and $\Orb(I)$ intersects $J$, then it need not be the case that $\Orb(I) \supseteq J$. However, we have the following.

\begin{lem}\label{lem:covering}
Let $f$ be a map of Sharkovskii type 1 and $\I$ its set of essential intervals. For any $\epsilon>0$ there exists $\delta>0$ such that for all $I,J\in\I$, if $\Orb(I)$ contains a point $x\in J$ with $|x-\source(J)|<\delta$, then either $\Orb(I)\supseteq J$ or $|\source(I)-\source(J)|<\epsilon$.
\end{lem}
\begin{proof}
If $I$ is an essential up interval, then by Lemma~\ref{lem:2up}, $\Orb(I)$ contains every essential up interval $J$ that it meets. If additionally the right endpoint $z$ of $\Orb(I)$ is a fixed point of $f$ or does not belong to any essential interval (i.e.\ it is outside the convex hull of the fixed points), then $\Orb(I)$ contains every essential interval that it meets, and there is nothing to prove. But if $z$ belongs to an essential down interval $J$, then $\Orb(I)$ does not contain $J$, so if $|\source(I)-\source(J)|\geq\epsilon$ then we must take care to choose $\delta<|z-\source(J)|$.

This motivates us to define a function $z:\I\to[0,1]$ by the rule
\begin{equation*}
z(I):=\begin{cases}
\sup(\Orb(I)), & \text{ if $I\in\I$ is an essential up interval},\\
\inf(\Orb(I)), & \text{ if $I\in\I$ is an essential down interval}.
\end{cases}
\end{equation*}
Now let $\I'\subseteq\I$ be the subset of those essential intervals whose orbits have diameter greater than or equal to $\epsilon/2$. We will show that the set $z(\I')$ is finite. Then it suffices to take a positive number $\delta<\epsilon/2$ which is less than the minimum of the finite set 
\begin{equation*}
\{|z(I)-\source(J)| ~:~ I\in\I', J\in\I, z(I)\in J\}
\end{equation*}
(it is a set of positive numbers, because $\source (J) \notin J$ by definition).
Then for two intervals $I,J\in\I$, if $\Orb(I)$ does not cover $J$ but contains a point in $J$ which is $\delta$-close to its source, then $z(I)\in J$, $|z(I)-\source(J)|<\delta<\frac\epsilon2$, and $I\not\in\I'$ by the definition of $\delta$. Therefore $|\source(I)-z(I)|<\frac\epsilon2$ and by the triangle inequality $|\source(I)-\source(J)|<\epsilon$.

It remains to prove that $z(\I')$ is finite. We give a proof by contradiction. Suppose $z(\I')$ is infinite. Without loss of generality we may suppose that $\I'$ contains infinitely many up intervals $I_n$, $n\in\mathbb{N}$, such that the points $z_n=z(I_n)$ are all distinct. Clearly the points $x_n=\source(I_n)$ are also all distinct. After passing to a subsequence we may assume that the numbers $x_n$ converge to some limit point $x\in[0,1]$ monotonically from one side. Then by passing to a tail of the sequence we may assume that $|x_n - x| < \epsilon/2$ for all $n$. Since the sets $\Orb(I_n)$ are all intervals of diameter at least $\epsilon/2$ with the points $x_n$ as their respective left endpoints, it follows that $\Orb(I_n)\cap\Orb(I_m)\neq\emptyset$ for $m,n\in\mathbb{N}$. Then by Lemma~\ref{lem:2up} the sets $\Orb(I_n)$ form a nested sequence: if $x_n\nearrow x$ is increasing, then $\Orb(I_1)\supset\Orb(I_2)\supset\cdots$ is decreasing, and conversely if $x_n\searrow x$ is decreasing, then $\Orb(I_1)\subset\Orb(I_2)\subset\cdots$ is increasing, see Figure~\ref{fig:2cases}. We address these two cases separately.

\begin{figure}[htb!!]
\begin{tikzpicture}
\draw[fill] (0,0) circle [radius=0.08] node[below=3pt] {$x_1$};
\draw[fill] (1,0) circle [radius=0.08] node[below=3pt] {$x_2$};
\draw[fill] (1.6,0) circle [radius=0.08] node[below=3pt] {$x_3$};
\draw[fill] (1.9,0) circle [radius=0.02];
\draw[fill] (2.0,0) circle [radius=0.02];
\draw[fill] (2.1,0) circle [radius=0.02];
\draw[fill] (2.3,0) circle [radius=0.08] node[below=3pt] {$x$};
\draw[fill] (7.7,0) circle [radius=0.08] node[below=3pt] {$z$};
\draw[fill] (7.9,0) circle [radius=0.02];
\draw[fill] (8.0,0) circle [radius=0.02];
\draw[fill] (8.1,0) circle [radius=0.02];
\draw[fill] (8.4,0) circle [radius=0.08] node[below=3pt] {$z_3$};
\draw[fill] (9.0,0) circle [radius=0.08] node[below=3pt] {$z_2$};
\draw[fill] (10,0) circle [radius=0.08] node[below=3pt] {$z_1$};
\draw[thick,gray,<->] (2.5,0) -- node[black,above] {\scriptsize$\geq\epsilon/2$} (7.5,0);
\node at (5,-1) {-- or --};
\begin{scope}[shift={(0,-2)}]
\draw[fill] (2.3,0) circle [radius=0.08] node[below=3pt] {$x_1$};
\draw[fill] (1.3,0) circle [radius=0.08] node[below=3pt] {$x_2$};
\draw[fill] (0.7,0) circle [radius=0.08] node[below=3pt] {$x_3$};
\draw[fill] (0.4,0) circle [radius=0.02];
\draw[fill] (0.3,0) circle [radius=0.02];
\draw[fill] (0.2,0) circle [radius=0.02];
\draw[fill] (0,0) circle [radius=0.08] node[below=3pt] {$x$};
\draw[fill] (10,0) circle [radius=0.08] node[below=3pt] {$z$};
\draw[fill] (9.8,0) circle [radius=0.02];
\draw[fill] (9.7,0) circle [radius=0.02];
\draw[fill] (9.6,0) circle [radius=0.02];
\draw[fill] (9.3,0) circle [radius=0.08] node[below=3pt] {$z_3$};
\draw[fill] (8.7,0) circle [radius=0.08] node[below=3pt] {$z_2$};
\draw[fill] (7.7,0) circle [radius=0.08] node[below=3pt] {$z_1$};
\draw[thick,gray,<->] (2.5,0) -- node[black,above] {\scriptsize$\geq\epsilon/2$} (7.5,0);
\end{scope}
\end{tikzpicture}
\caption{Nested orbits of essential up intervals: two cases}\label{fig:2cases}
\end{figure}
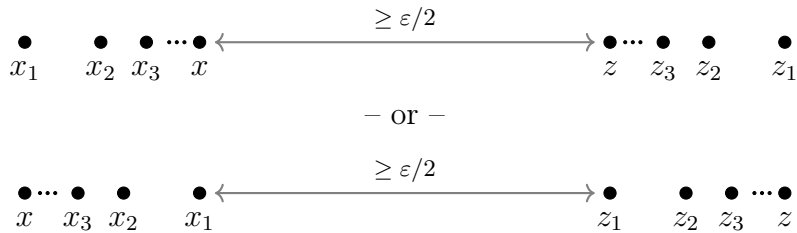

First suppose $x_n\nearrow x$ is increasing. Then the numbers $z_n$ form a decreasing sequence and hence converge to their infimum $z=\inf z_n$. Since each $I_n \in \I'$ we get $|z-x| \geq \epsilon/2$. Note that each $x_n$ is a fixed point of $f$ and so $(x_n,x_{n+1})$ contains the essential interval $I_n$. It follows that $\Orb((x_n,x_{n+1})) \supseteq \Orb(I_n) \supseteq (z_{n+1},z_n)$. But the set $[x_{n+1},z_{n+1}]=\overline{\Orb(I_{n+1})}$ is invariant for $f$. Therefore $f((x_n,x_{n+1}))$ must contain $(z_{n+1},z_n)$ for all $n$. Since the points $x_n$ are fixed points this violates the uniform continuity of $f$.

Now suppose $x_n\searrow x$ is decreasing. Then the numbers $z_n$ form an increasing sequence and hence converge to their supremum $z=\sup z_n$. Since $I_1\in \I'$ we have $|z_1-x_1| \geq \epsilon/2$. The rest of the proof is analogous to the previous case: we observe that $(x_{n+1},x_n)$ contains $I_{n+1}$ and $[x_n,z_n]$ is invariant, so that $f((x_{n+1},x_n))$ must contain $(z_n,z_{n+1})$ for all $n$, leading to a contradiction with the fact that the points $x_n$ are fixed and $f$ is uniformly continuous.
\end{proof}

\begin{lem}\label{lem:find-ess-int}
Let $f$ be a map of Sharkovskii type 1 and $\I$ its set of essential intervals. For every $\epsilon>0$ there exists $\delta>0$ such that for $x\in[0,1]$, $c\in\Fix(f)$, $t\in\N$, if $|x-c|<\delta$ and $|f^t(x)-c|\geq\epsilon$, then there exist $I\in\I$ and $\tau\in\Nzero$, $\tau<t$, such that $f^\tau(x)\in I$ and $|\source(I)-c|<\epsilon$ and $|\source(I)-f^\tau(x)|<\epsilon$.
\end{lem}
\begin{proof}
Use the uniform continuity of $f$ to find a positive number $\delta<\epsilon$ such that points which are $\delta$-close have images which are $\epsilon$-close. We may suppose without loss of generality that $c-\delta<x<c$ and that $t\in\N$ is the smallest natural number with $|f^t(x)-c|\geq\epsilon$. Since $c$ is fixed the point $f^{t-1}(x)$ cannot belong to the ball $B(c,\delta)$, and in particular we have $t\geq2$.

Let $J$ be the connected component of $[0,1]\setminus\Fix(f)$ containing $x$. We do not yet know if $J$ is an essential interval, since we do not know if its left endpoint is fixed. However, if each point in $J$ moves to the left, then its left endpoint must be fixed, so $J$ is an essential down interval and $x<\source(J)<c$. In this case we may take $I=J$ and $\tau=0$ and we are finished.

Henceforth we suppose that each point in $J$ moves to the right. Then by Lemma~\ref{lem:sharkovskii} the whole orbit of $x$ lies to the right of $x$, so we must have
\begin{equation*}
c+\delta\leq f^{t-1}(x) < c+\epsilon \leq f^t(x),
\end{equation*}
where the middle inequality uses the minimality of $t$. Let $I$ be the connected component of $[0,1]\setminus\Fix(f)$ containing $f^{t-1}(x)$ and let $\tau=t-1$. Since $f^{t-1}(x)$ moves to the right, every point of $I$ must move to the right, so in particular the right endpoint of $I$ must be fixed. Then $I$ must be an essential up interval whose left endpoint (the source) satisfies $c\leq \source(I) < f^{t-1}(x) = f^\tau(x) < c+\epsilon$.
\end{proof}

The following proposition completes the proof of Theorem~\ref{T:type1}.
\begin{prop}\label{P:ell-chain}
Let $f$ be an interval map of Sharkovskii type 1 and let $\ell$ be a positive integer. If $\hpol(f)>\ell-1$, then $f$ has an $\ell$-chain of essential intervals.
\end{prop}

\begin{proof}
Fix $\epsilon>0$ small enough so that $\hpol^{2\epsilon}(f)>\ell-1$. We will define an open cover $\U_\epsilon$ of $[0,1]$ by balls of radius $\epsilon$. We will code the trajectories of points using the elements of $\U_\epsilon$ in a way which allows us to detect an $\ell$-chain of essential intervals for $f$.

\smallskip

\emph{Step 1: Construction of $\U_\epsilon$}.\\
Using Lemma~\ref{lem:never-return} find a positive number $\delta<\epsilon$ such that for $x\in[0,1]$, $c\in\Fix(f)$, $0\leq s<t$,
\begin{equation}\label{1a}
\begin{gathered}
\text{If } |f^s(x)-c|<\delta \text{ and } |f^t(x)-c|\geq\epsilon,\\
\text{then } |f^m(x)-c|\geq\delta \text{ for all } m\geq t.
\end{gathered}
\end{equation}
Using Lemma~\ref{lem:covering} find a positive number $\gamma<\frac{\delta}{6}$ such that for $x\in[0,1]$, $I,I'\in\I$, $0\leq\tau<\tau'$,
\begin{equation}\label{1b}
\begin{gathered}
\text{If } f^\tau(x)\in I \text{ and } f^{\tau'}(x)\in I' \text{ and } |f^{\tau'}(x)-\source(I')|<\gamma,\\
\text{then } |\source(I)-\source(I')|<\frac{\delta}{2} \text{ or } \Orb(I)\supseteq I'.
\end{gathered}
\end{equation}
Using Lemma~\ref{lem:find-ess-int} find a positive number $\beta<\gamma$ such that for $x\in[0,1]$, $c\in\Fix(f)$, $0\leq s<t$,
\begin{equation}\label{1c}
\begin{gathered}
\text{If } |f^s(x)-c|<\beta \text{ and } |f^t(x)-c|\geq\gamma,\\
\text{then there exist } \tau\in\Nzero, I\in\I \text{ such that } s\leq\tau<t \text{ and } f^\tau(x)\in I\\
\text{and } |\source(I)-c|<\gamma \text{ and } |\source(I)-f^\tau(x)|<\gamma.
\end{gathered}
\end{equation}
Now choose a finite subset $C\subseteq\Fix(f)$ such that $\bigcup_{c\in C} B(c,\beta)$ covers the compact set $\Fix(f)$. By Lemma~\ref{lem:sharkovskii}, all non-wandering points of $f$ are fixed. Thus the compact set $[0,1]\setminus\bigcup_{c\in C} B(c,\beta)$ contains only wandering points of $f$. Choose a finite set $W$ of these wandering points with corresponding radii $r_w<\epsilon$, $w\in W$, such that each ball $B(w,r_w)$ is wandering in the sense that $B(w,r_w) \cap f^i(B(w,r_w)) = \emptyset$ for all $i\geq 1$ and such that the union of all the balls $B(w,r_w)$, $w\in W$, together with all the balls $B(c,\beta)$, $c\in C$, covers the whole interval $[0,1]$. Let $A=W\cup C$ and put $\U_\epsilon = \left\{ B(a,\epsilon) ~;~ a \in A\right\}$. Then $\U_\epsilon$ is an open cover of $[0,1]$ by balls of uniform radius $\epsilon$ which we will use to estimate the $2\epsilon$-polynomial entropy.

\smallskip

\emph{Step 2: Coding trajectories with respect to $\U_\epsilon$}.\\
Since $\U_\epsilon$ is an open cover of $[0,1]$ every point $x\in[0,1]$ has at least one itinerary with respect to $\U_\epsilon$, i.e.\ a sequence of balls $B(a_n,\epsilon)$, $n=0,1,\ldots$, such that $a_n\in A$ and $f^n(x)\in B(a_n,\epsilon)$ for all $n$. The sequence of centers $a_0 a_1 a_2 \cdots$ of these balls will be called a \emph{code} of $x$. Since $\U_\epsilon$ is an open cover rather than a partition, a point in general has many codes. For each $x\in[0,1]$ we will construct one code which has a block structure of the form
\begin{equation}\label{blocks}
\star\ \framebox{$c_1\ldots c_1$\strut}\ \star\ \framebox{$c_2\ldots c_2$\strut}\ \star\ \cdots\ \star\  \framebox{$c_{\lambda} \ldots c_{\lambda}$\strut}\ \star\ \framebox{$c_{\lambda+1} c_{\lambda+1} \ldots$\strut},
\end{equation}
where each $\star$ represents a finite (possibly empty) block of symbols from $W$, each symbol from $W$ is used at most once in the whole code, $\lambda=\lambda(x)\geq 0$, the constant blocks of the form $c_i \ldots c_i$ are nonempty, the last constant block $c_{\lambda+1} c_{\lambda+1} \ldots$ has infinite length, each symbol $c_i$ is an element of $C$, and the points $c_1, c_2, \ldots, c_{\lambda+1}$ are pairwise distinct. Note that $\lambda(x)$ counts the number of \emph{switches} from one constant block to another in the code of $x$.

Now we describe a procedure to choose a code of this form for any point $x\in[0,1]$.
\begin{itemize}
\item Since by Lemma~\ref{lem:sharkovskii} the trajectory of $x$ converges to a fixed point, there is a smallest $s_1\geq 0$ such that $f^{s_1}(x) \in \bigcup_{c\in C} B(c,\beta)$. Choose $c_1 \in C$ such that $f^{s_1}(x) \in B(c_1,\beta)$.
\item If $f^{s_1}(x)$ never leaves the larger ball $B(c_1,\epsilon) \in \U_\epsilon$, then put $\lambda=0$ (there will be no switches between constant blocks) and write the symbol $c_1$ in the code of $x$ in positions $s_1, s_1+1,\ldots$.
\item Otherwise, there is a smallest $t_1> s_1$ such that $f^{t_1}(x)\not\in B(c_1,\epsilon)$. Write the symbol $c_1$ in the code of $x$ in positions $s_1, s_1+1, \ldots, t_1-1$.
\item Now continue inductively. If $c_{i-1},s_{i-1},t_{i-1}$ have already been defined, then since the trajectory of $x$ converges to a fixed point, there is a smallest $s_i\geq t_{i-1}$ such that $f^{s_i} \in \bigcup_{c\in C} B(c,\beta)$. Choose $c_i \in C$ such that $f^{s_i}(x) \in B(c_i,\beta)$.
\item If $f^{s_i}(x)$ never leaves the larger ball $B(c_i,\epsilon)$, then put $\lambda=i-1$ and write the symbol $c_i$ in the code of $x$ in positions $s_i, s_i+1, \ldots$.
\item Otherwise, there is a smallest $t_i>s_i$ such that $f^{t_i}(x)\not\in B(c_i,\epsilon)$. Write the symbol $c_i$ in the code of $x$ in positions $s_i, s_i+1, \ldots, t_i-1$, and repeat the induction step again.
\end{itemize}
Now we show that any two symbols $c_i, c_j$, $i<j$, chosen in this coding procedure are distinct. We have $|f^{s_i}(x)-c_i|<\beta<\delta$ and $|f^{t_i}(x)-c_i|\geq\epsilon$, so by~\eqref{1a} the trajectory never returns to $B(c_i,\delta)$ after time $t_i$. But $s_j\geq t_i$ so in particular $f^{s_j}(x)\not\in B(c_i,\delta)$. On the other hand $f^{s_j}(x)\in B(c_j,\beta)$, so by the triangle inequality $\delta \leq |f^{s_j}(x) - c_i| \leq |f^{s_j}(x) - c_j| + |c_j-c_i| < \beta + |c_j-c_i|$, whence
\begin{equation}\label{ci-cj}
|c_i-c_j| > \delta - \beta > 0.
\end{equation}
Now that we know that the symbols chosen from $C$ are distinct, it follows that $\lambda$ is finite because $C$ is a finite set. This shows that the code of $x$ has the form~\eqref{blocks}, where we still need to fill in the $\star$-blocks. Continue the coding procedure as follows:
\begin{itemize}
\item If the positions occupied by the symbols $c_1, \ldots, c_{\lambda+1}$ cover the whole set $\Nzero$, then we are done.
\item Otherwise, if the $m$th position in the code of $x$ has not yet been filled, then by our coding procedure $f^m(x)\not\in \bigcup_{c\in C} B(c,\beta)$. Therefore $f^m(x) \in \bigcup_{w\in W} B(w,r_w)$. So choose $w\in W$ such that $f^m(x) \in B(w,r_w)$ and write the symbol $w$ in position $m$.
\end{itemize}
The fact that the sets $B(w,r_w)$ are wandering guarantees that we do not use any symbol $w\in W$ more than once. This concludes the construction of a code for $x$ of the form~\eqref{blocks}.

\smallskip

\indent\emph{Step 3: Finding a point $x$ with at least $\ell$ switches in its code}.\\
Suppose that $\lambda(x)\leq\ell-1$ for all $x\in[0,1]$. We will use this assumption to show that $\hpol^{2\epsilon}(f)\leq\ell-1$, contradicting the choice of $\epsilon$ in Step 1. Recall that $\hpol^{2\epsilon}(f)$ is calculated using the separation numbers $\sep(n,2\epsilon,f)$. Observe that if $x,y$ are $(n,2\epsilon,f)$-separated, then their codes in step 2 differ in at least one of the first $n$ positions. This shows that $\sep(n,2\epsilon,f)$ is bounded above by the number of distinct length-$n$ initial segments from the set of all codes chosen in step 2. Each of those initial segments is a word $\a\in A^n$ with the following two properties:
\begin{enumerate}
\item If some symbol from $A$ is repeated more than once in $\a$, then all these repetitions occur together in a single block (call it a constant block), and
\item there are at most $\ell$ constant blocks in $\a$, i.e.\ at most $\ell-1$ switches between constant blocks (since we assumed $\lambda(x)\leq\ell-1$ for each $x$).
\end{enumerate}
Any word with those two properties will be called \emph{allowable}, and we wish to count the number $\omega(n)$ of allowable length-$n$ words as it gives an upper bound on $\sep(n,2\epsilon,f)$ (the upper bound is quite rough, but it is sufficient for our purposes). Much like with the coded trajectories, we picture an allowable word $\a$ as an alternating concatenation of $\star$-blocks and constant blocks
\begin{equation*}
\a = \underbrace{\star \ \framebox{$a_1 \ldots a_1$\strut} \ \star \ \framebox{$a_2 \ldots a_2$\strut} \ \star \ \cdots \ \star \ \framebox{$a_{m} \ldots a_{m}$\strut} \ \star}_{\text{length }n}
\end{equation*}
where the $\star$-blocks (they can be empty) consist of the symbols that are used only once in $\a$ and the number $m$ of constant blocks is at most $\ell$. We claim that for $n>\# A$,
\begin{equation}\label{words}
\omega(n) \leq
\underbracket{\sum\limits_{i=1}^{\#A}\vphantom{\sum\limits_{l=0}^{\min(i,S+1)}}}_{(1)}
\underbracket{\binom{\#A}{i}\vphantom{\sum\limits_{l=0}^{\min(i,S+1)}}}_{(2)}
\underbracket{\left(i!\right)\vphantom{\sum\limits_{l=0}^{\min(i,S+1)}}}_{(3)}
\underbracket{\sum\limits_{m=1}^{\min(i,\ell)}}_{(4)}
\underbracket{\binom{i}{m}\vphantom{\sum\limits_{m=1}^{\min(i,L+1)}}}_{(5)}
\underbracket{p_2(m,n-(i-m))\vphantom{\sum\limits_{m=1}^{\min(i,L+1)}}}_{(6)},
\end{equation}
where $p_2(m,n-(i-m))$ is the number of ordered $m$-tuples of natural numbers greater than or equal to 2 whose sum is $n-(i-m)$. The justification for~\eqref{words} is that each allowable word $\a$ is uniquely determined by the following choices.
\begin{enumerate}
\item Choose the number $i$ of distinct letters used in $\a$.
\item Choose which $i$ letters from the alphabet $A$ to use.
\item Choose the order that those letters appear.
\item Choose the number $m$ of constant blocks. (Since $n>\# A$, we have $m\geq 1$. On then other hand, there are at most $\ell$ constant blocks in $\a$, and $m\leq i$ is obvious.)
\item Choose which $m$ of the $i$ letters to use in those constant blocks.
\item Choose the lengths of those blocks to add up to $n-(i-m)$. Since $i-m$ letters are used in the $\star$-blocks, this gives the word $\a$ the desired length $n$.
\end{enumerate}
Notice especially the constraint that the lengths of the constant blocks must add up to $n-(i-m)$, where $i,m$, are bounded by constants that do not depend on $n$. If we let $p_0(m,n)$ denote the number of ordered $m$-tuples of nonnegative integers whose sum is $n$, then clearly
\begin{equation*}
p_2(m,n-(i-m)) \leq p_0(m,n-(i-m)) \leq p_0(m,n) \leq p_0(\ell,n)
\end{equation*}
whenever $m \leq \ell$. By the stars and bars theorem from combinatorics, $p_0(\ell,n)=\binom{n+\ell-1}{\ell-1}$. So for some constant $K$ which depends on $\ell$ and $\# A$ but not $n$, we have\footnote{The upper estimate $\omega(n)\leq \textrm{Constant}\cdot n^{\ell-1}$ could seem at first counterintuitive based on the following reasoning: some of the words are allowed to have $\ell$ constant blocks, and it seems that there are at least $\binom{n}{\ell}$ such words because this is the number of choices of the starting positions of $\ell$ blocks, and $\binom{n}{\ell} \sim n^{\ell}$ (even if we take into account that those blocks have lengths at least $2$, that order is still $n^{\ell}$). But this argument is wrong because we cannot choose the starting positions of the constant blocks arbitrarily. Indeed, in each allowable word the $\star$-blocks must be short when $n \gg \#A$, because each symbol in a $\star$-block can be used only once. This reflects the fact that outside $\bigcup_{c\in C} B(c,\beta)$ the graph of $f$ is far away from the diagonal and so the trajectory moves quickly from one constant block to the next one. In particular, the beginning position of the first constant block is rather close to zero. Even if the beginnings of the other $\ell-1$ constant blocks can occur in more or less arbitrary positions, this gives only $\binom{n}{\ell-1}$ choices. Therefore it is not surprising that $\omega(n)$ is of the order $n^{\ell-1}$ rather than $n^{\ell}$.}
\begin{equation*}
\omega(n) \leq K \binom{n+\ell-1}{\ell-1}.
\end{equation*}
This is a polynomial in $n$ of degree $\ell-1$ and therefore
\begin{equation*}
\hpol^{2\epsilon}(f) = \limsup_{n\to\infty} \frac{\log \sep(n,2\epsilon,f)}{\log n} \leq \limsup_{n\to\infty} \frac{\log\omega(n)}{\log n} \leq \ell-1.
\end{equation*}
This contradiction shows that there must be some point $x$ with $\lambda(x)\geq\ell$.

\smallskip

\emph{Step 4: Construction of a chain of essential intervals of length $\ell$.}\\
Fix $x\in[0,1]$ with $\lambda(x)\geq\ell$ and consider its code of the form~\eqref{blocks} constructed in Step 2. We will use the orbit of $x$ to find distinct essential intervals $I_1, \ldots, I_{\ell}$ of $f$ with sources close to the points $c_1, \ldots, c_{\ell}$, respectively, and forming a chain of essential intervals. It is critical that the code makes at least $\ell$ switches between constant blocks, and in particular the block $c_{\ell}\cdots c_{\ell}$ eventually terminates (at time $t_\ell$, when the orbit moves far enough away from $c_\ell$).

For each $i=1,\ldots,\ell$ consider the behavior of the trajectory of $x$ from time $s_i$ to $t_i$. By the coding procedure we have $|f^{s_i}(x)-c_i|<\beta$ and $|f^{t_i}(x)-c_i|\geq\epsilon>\gamma$. By~\eqref{1c} there is a time $\tau_i$ and an essential interval $I_i$ such that $s_i\leq \tau_i < t_i$ and
\begin{equation}\label{Ii}
f^{\tau_i}(x)\in I_i \text{ and } |\source(I_i)-c_i|<\gamma \text{ and } |\source(I_i)-f^{\tau_i}(x)|<\gamma,
\end{equation}
see Figure~\ref{fig:codes}.

\begin{figure}[htb!!]
	\begin{tikzpicture}
		\draw[thick] (-4,0) -- (7,0);
		\draw[thick] (0,.2) -- (0,-.4) node[below] {$c_i$};
		\draw (-0.25,.45) -- (0.25,.45) node[right] (beta) {\scriptsize $B(c_i,\beta)$};
		\node (betatext) at (8,.45) {\scriptsize visited at the start of the $i$th constant block};
		\draw[dotted] (beta) -- (betatext);
		\draw (-0.5,.8) -- (0.5,.8) node[right] (gamma) {\scriptsize $B(c_i,\gamma)$};
		\node (gammatext) at (8,.8) {\scriptsize contains $f^{\tau_i}(x)$ and $\source(I_i)$};
		\draw[dotted] (gamma) -- (gammatext);
		\draw (-1.5,1.15) -- (1.5,1.15) node[right] (delta) {\scriptsize $B(c_i,\delta)$};
		\node (deltatext) at (8,1.15) {\scriptsize never visited again after the block ends};
		\draw[dotted] (delta) -- (deltatext);
		\draw (-2.5,1.5) -- (2.5,1.5) node[right] (epsilon) {\scriptsize $B(c_i,\epsilon)$};
		\node (epsilontext) at (8,1.5) {\scriptsize exited at the end of the block};
		\draw[dotted] (epsilon) -- (epsilontext);
		\draw[fill] (-.15,0) circle [radius=0.05] node[below] {\scriptsize $s_i$};
		\draw[fill] (0.4,0) circle [radius=0.05] node[below] {\scriptsize $\tau_i$};
		\draw[fill] (4.5,0) circle[radius=0.05] node[below] {\scriptsize $t_i$};
%		\draw[fill,gray] (2.2,0) circle[radius=0.05] node[below] {\scriptsize $s_j$};
%		\draw[gray,thick] (2.0,.2) -- (2.0,-.4) node[below] {$c_j$};
%		\draw[fill,gray] (1.6,0) circle[radius=0.05] node[below] {\scriptsize $\tau_j$};
		\draw (0.2,-1) -- node[below] {\scriptsize $I_i$} (0.8,-1);
		\draw (0.4,-1) + (135:0.2) -- +(180:0.2) -- + (215:0.2);
%		\draw[gray] (1.8,-1) -- node[below] {\scriptsize $I_j$} (1.2,-1);
%		\draw[gray] (1.6,-1) + (45:0.2) -- +(0:0.2) -- + (-45:0.2);
	\end{tikzpicture}
	\caption{Balls around fixed points used in the coding procedure.}\label{fig:codes}
\end{figure}
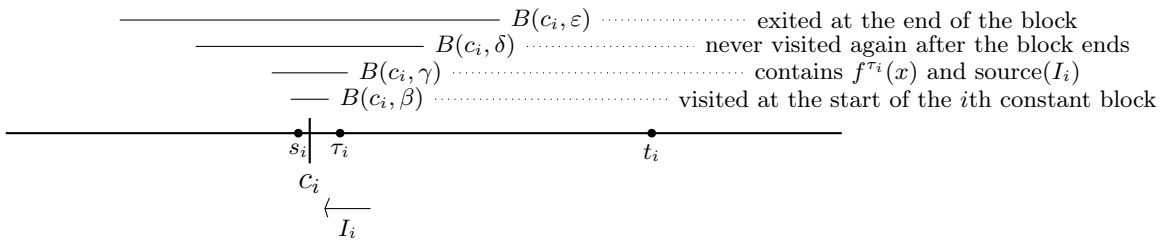

Now we show that the essential intervals $I_1, I_2, \ldots, I_{\ell}$ are pairwise distinct. Let $1\leq i<j\leq \ell$ and consider the two times $\tau_i, \tau_j$. Clearly $\tau_i<t_i\leq s_j<\tau_j$. Using the triangle inequality with~\eqref{ci-cj} and~\eqref{Ii} we get
\begin{equation}\label{Ii-Ij}
\begin{aligned}
|\source(I_i)-\source(I_j)| &\geq |c_i-c_j| - |\source(I_i)-c_i| - |\source(I_j)-c_j|\\
&\geq\delta-\beta-2\gamma > \delta -3\gamma > \frac{\delta}{2}, 
\end{aligned}
\end{equation}
since we chose $\beta<\gamma<\frac{\delta}{6}$. This shows that $I_i, I_j$ have distinct sources, hence they are distinct essential intervals.

Finally whenever $1\leq i < j \leq \ell$ we have $\tau_i<\tau_j$ and it follows from~\eqref{1b}, \eqref{Ii}, and \eqref{Ii-Ij} that $\Orb(I_i) \supseteq I_{j}$. This completes the proof that $I_1, I_2, \ldots, I_{\ell}$ is a chain of essential intervals for $f$.
\end{proof}

%%%%%%%%%%%%%%%%%%%%%%%%%%%%%%%%%%%%%%%%%%%%%%%%%%%%%%%%%%%%%%%%%%%%%%%%%%%%%%%%%%%%%

\subsection{Simple $2^n$-cycles}\label{SS:simple}

%%%%%%%%%%%%%%%%%%%%%%%%%%%%%%%%%%%%%%%%%%%%%%%%%%%%%%%%%%%%%%%%%%%%%%%%%%%%%%%%%%%%%

Our next goal is to extend Theorem~\ref{T:type1} to interval maps of arbitrary Sharkovskii type. Recall the well-known Sharkovskii order $\prec$ on the set $\mathbb{N}\cup\{2^\infty\}$ of natural numbers (positive integers) with an added symbol $2^\infty$, given as follows:
\begin{multline*}
1 \prec 2 \prec 4 \prec 8 \prec \cdots \prec 2^\infty \prec \\
\cdots \prec
\cdots \prec 2^2\cdot7 \prec 2^2\cdot5 \prec 2^2\cdot3
\prec
\cdots \prec 2\cdot7 \prec 2\cdot5 \prec 2\cdot3
\prec
\cdots \prec 7 \prec 5 \prec 3
\end{multline*}
Sharkovskii proved that if an interval map $f$ has a cycle of period $n\in\mathbb{N}$, then it also has cycles of all periods $m\in\mathbb{N}$ which precede $n$ in the Sharkovskii order~\cite{Sh64}.
The \emph{Sharkovskii type} of an interval map $f$ is the $\prec$-maximum of the periods of its cycles, or $2^\infty$ if the set of periods is the set of powers of $2$. We will write $\Type(f)\in\N\cup\{2^\infty\}$ for the Sharkovskii type. We will also use the Sharkovskii order to compare types, for example, to speak of maps of type greater than or equal to $2^\infty$.

We are especially interested in maps of Sharkovskii type $2^n$ or $2^\infty$, since in any other case $f$ has positive topological entropy and therefore infinite polynomial entropy. To simplify notation, if $P=\{x_1,\cdots,x_{2^n}\}$ is a $2^n$-cycle of $f$ with the spatial ordering $x_1 < \cdots < x_{2^n}$, we will express this compactly by writing $P=\{x_1<\cdots<x_{2^n}\}$.

Recall that a cycle $P$ of period $2^k$, $k\geq 0$, is called \emph{simple} if it can be obtained from a cycle of period $1$ by making $k$-times $2$-extensions (in the terminology of \cite{ALM00}). An equivalent definition is by induction \cite{Bl79}. All cycles of periods $1$ and $2$ are simple. If $n\geq 1$, a $2^{n+1}$-cycle $P=\{x_1 < x_2 < \cdots < x_{2^{n+1}}\}$ of $f$ is simple if for $P_{L}=\{x_1, \dots, x_{2^n}\}$ and $P_{R}=\{x_{2^n+1}, \dots, x_{2^{n+1}}\}$ we have $f(P_L)=P_R$, $f(P_R)=P_L$ and $P_L$, $P_R$ are simple $2^n$-cycles of $f^2$. If $n\geq 1$ and $P$ is a simple $2^{n+1}$-cycle of $f$ then, in the notation used above, the interval $M=[\max P_L, \min P_R]$ will be called the \emph{middle interval of $P$}. Note that, since $n\geq 1$, the middle interval $M$ of $P$ lies ``strictly" inside the convex hull $\conv (P)$ in the sense that $\min P < \min M < \max M < \max P$.

For our purposes it is important that if $f$ has a $2^n$-cycle then it also has a simple $2^n$-cycle, see~\cite{BH83} or \cite[Corollary 2.11.2]{ALM00}. 

Recall also the general fact that for a map $f$ on a compact metric space, each iterate $f^n$ has the same set of recurrent points as $f$. In particular, if we speak about a non-recurrent point, it is not a problem if we forget to specify whether it is non-recurrent with respect to $f$ or some iterate.

\begin{lem}\label{L:middle}
Let $\varphi\colon [0,1]\to [0,1]$ have a simple $2^n$-cycle $P$ with $n\geq 2$ and middle interval~$M$. Then $\varphi^{2^n}(M) \supseteq \conv (P)$ and $M$ contains a non-recurrent point.	
\end{lem}

\begin{proof}
We have $P=\{x_1 < \cdots < x_{2^{n}}\}$, $P_{L}=\{x_1, \dots, x_{2^{n-1}}\}$ and $P_{R}=\{x_{2^{n-1} +1}, \dots, x_{2^{n}}\}$, $M=[x_{2^{n-1}}, x_{2^{n-1}+1}]$. Since the left endpoint of $M$ is mapped to $P_R$ and the right endpoint of $M$ is mapped to $P_L$, we have $\varphi (M) \supseteq M$. hence
\[
M \subseteq \varphi (M)	\subseteq \varphi^2 (M)	 \subseteq \cdots \subseteq \varphi^{2^n} (M).	
\]
Since $M$ contains a point from the $2^n$-cycle $P$, the union of these sets contains the whole set~$P$. It follows that $\varphi^{2^n} (M) \supseteq P$ and, by intermediate value theorem, 
$\varphi^{2^n} (M) \supseteq \conv (P)$. Any point $x\in M$ with $\varphi^{2^n}(x)=x_1$ is a non-recurrent point of $\varphi$.
\end{proof}

Let us agree to say that a one-way horseshoe $A_1, \dots, A_{\ell}$ lies strictly in the closed interval $[a,b]$ if $A_1 \cup \dots \cup A_{\ell} \subseteq (a,b)$.

\begin{prop}\label{P:simple-horse}
let $f\colon [0,1] \to [0,1]$ have a simple $2^n$-cycle $P$, $n\geq 2$. Then $f^{2^n}$ has a one-way $(n-1)$-horseshoe composed of pairwise disjoint closed intervals $A_1, \cdots, A_{n-1}$ strictly contained in $\conv (P)$.
\end{prop}

\begin{proof}
By induction on $n$.

Let $n=2$.	Then $P=\{x_1 < x_2 < x_3 < x_4\}$, $M=[x_2, x_3]$ and, by Lemma~\ref{L:middle}, $f^4(M) \supseteq [x_1, x_4] \supseteq M$ and $M$ contains a non-recurrent point. Therefore $M$ is a one-way horseshoe of order $1$ for $f^4$.

Now assume that the claim from the proposition holds for some $n\geq 2$. We prove that it holds for $n+1$. Let $f$ have a simple $2^{n+1}$-cycle $P=\{x_1 < \cdots < x_{2^{n+1}}\}$, $P_{L}=\{x_1, \dots, x_{2^n}\}$, $P_{R}=\{x_{2^n+1}, \dots, x_{2^{n+1}}\}$, $M=[x_{2^n}, x_{2^n+1}]$. Put $g=f^2$. Then $P_R$ is a simple $2^n$-cycle of $g$ and so, by the induction hypothesis, $g^{2^n}$ has a one-way $(n-1)$-horseshoe composed of disjoint closed intervals $A_1, \dots, A_{n-1}$ lying strictly in $\conv (P_R)$. In particular, all the intervals in this horseshoe are disjoint from $M$, and each of them contains a non-recurrent point. However, $g^{2^n}=f^{2^{n+1}}$ and, by Lemma~\ref{L:middle}, $f^{2^{n+1}} (M)\supseteq \conv (P) \supseteq M$ and $M$ contains a non-recurrent point. Therefore $M, A_1, \dots, A_{n-1}$ is a one-way horseshoe for $f^{2^{n+1}}$ of order $n$, which lies strictly in $\conv (P)$.
\end{proof}

\subsection{Main rigidity theorem}\label{SS:rigid}

Here is the promised analogue of Misiurewicz's theorem, that polynomial entropy is given by one-way horseshoes. We give it the name ``rigidity theorem,'' because it also answers in the affirmative the main question we were interested in, namely whether the polynomial entropy on the interval takes only integer values (and the value $\infty$).

In view of condition (d) in Theorem~\ref{T:type1} it is not surprising that, similarly as in the case of topological entropy and horseshoes, also in the theory of polynomial entropy of interval maps of all Sharkovskii types it is sufficient to consider one-way horseshoes made of closed intervals. 

Recall that throughout the paper the supremum of the empty set is zero, so for example if $f$ and its iterates have no one-way horseshoes, then the quantities in part (b) and (c) of the next theorem are zero. 

\begin{thm}[Rigidity theorem]\label{T:rigid}
For an interval map $f$, the following three quantities are equal (and they belong to $\mathbb N_0 \cup \{\infty\}$).
\begin{enumerate}
	\item [(a)] The polynomial entropy $\hpol (f)$.
	\item [(b)] The supremum of the orders of one-way horseshoes of $f$ and its iterates.
	\item [(c)] The supremum of the orders of one-way horseshoes of $f$ and its iterates consisting of (closed) intervals. 
\end{enumerate}	
In particular, if the polynomial entropy is finite, then it is an integer.
\end{thm}
\begin{proof}
Let $f\colon [0,1]\to [0,1]$. Let $S$ or $S_{\rm int}$ be the supremum of the numbers $\ell$ such that some iterate of $f$ has a one-way $\ell$-horseshoe formed of $\ell$ pairwise disjoint compact sets or pairwise disjoint compact intervals, respectively, and set $S=0$ or $S_{\rm int}=0$ if there are no such horseshoes. By Theorem~\ref{T:hor-ent} we have
\begin{equation}\label{Eq:geq S}
\hpol(f) \geq S \geq S_{\rm int}.
\end{equation}
It remains to prove that $\hpol(f) \leq S_{\rm int}$.

If $f$ has Sharkovskii type 1, then by Theorem~\ref{T:type1} we already have equality $\hpol(f)=S_{\rm int}$.

If $f$ has Sharkovskii type $2^{\infty}$ or higher, then it has cycles of period $2^n$ for all $n$. As already mentioned in Subsection~\ref{SS:simple},  it has also simple cycles of such periods. Thus by Proposition~\ref{P:simple-horse}, the iterates of $f$ have one-way horseshoes, composed of pairwise disjoint closed intervals, of arbitrarily large orders $\ell$. Therefore $S_{\rm int} =\infty$ and in view of~\eqref{Eq:geq S} we again have equality $\hpol(f)=S_{\rm int}$.

If $f$ has Sharkovskii type $2^n$ for some $n\geq 1$, then put $g=f^{2^n}$. Then every periodic point for $f$ is fixed for $g$ and since $f,g$ have the same sets of periodic points, the Sharkovskii type of $g$ is $1$. Since each iterate of $g$ is an iterate of $f$, the supremum $S'_{\rm int}$ of orders of one-way horseshoes, composed of pairwise disjoint closed intervals, for $g$ and its iterates is less than or equal to $S_{\rm int}$. Using Theorem~\ref{T:type1} and the power rule for polynomial entropy we have
\[
\hpol(f) =\hpol(f^{2^n}) = \hpol(g) = S'_{\rm int} \leq S_{\rm int}.
\]

We have thus shown that $\hpol(f)=S = S_{\rm int}$ for all interval maps and the proof is finished.
\end{proof}

\begin{rem}
Topological entropy of interval maps is lower semi-continuous with respect to uniform metric, see \cite{Mis79} or \cite[Theorem 4.5.2]{ALM00} and this fact is proved using Misiurewicz's theorem. In contrast with this, polynomial entropy of maps $[0,1] \to [0,1]$ is not lower semicontinuous. For instance, consider the map in the upper right corner of Figure~\ref{F:gs}. It has Sharkovskii type $1$ and infinite polynomial entropy (because it has an `infinite chain' of essential intervals converging to $1$). However, by an arbitrarily small perturbation we can get a map above the diagonal, except the fixed point $1$. Such a map has zero polynomial entropy.
\end{rem}

\begin{rem}
Recall that Misiurewicz's theorem on the relation between topological entropy and horseshoes was extended, with appropriately defined horseshoes on graphs, to graph maps in\cite{LM93}. It would be interesting to know if the same can be done for polynomial entropy. Therefore we propose the following questions.
\end{rem}

\begin{Q}
Is it still true that the polynomial entropy on graphs is given by one-way horseshoes (composed of compact sets)? Can the notion of a one-way horseshoe composed of closed intervals be generalized to graph maps in such a way as to give the polynomial entropy? 
\end{Q}

Finally we note that on dendrites the polynomial entropy cannot be given by one-way horseshoes, since it can take non-integer values, as we will see in Section~\ref{SS:dend} below.

\subsection{Possible values of polynomial entropy of interval maps}

\begin{thm}\label{T:hpol0}
The polynomial entropy of an interval map is zero if and only if its set of periodic points is connected. (Moreover, such a map is necessarily of Sharkovskii type 1 or 2.)
\end{thm}
\begin{proof}
The result holds for maps of Sharkovskii type 1 by Corollary~\ref{C:zero}.

If $f$ has Sharkovskii type 2, then put $g=f^2$. Then $f$ and $g$ have the same set of periodic points, and for $g$ these points are all fixed, so by Corollary~\ref{C:zero} this set is connected if and only if $\hpol(g)=0$. But $\hpol(f)=\hpol(g)$ by the power rule for polynomial entropy.

Finally, suppose that $f$ has any other Sharkovskii type. By Sharkovskii's theorem, $f$ has a cycle of period 4, and so  it has a simple 4-cycle $P$. By Proposition~\ref{P:simple-horse} there is a one-way horseshoe of order 1 strictly contained in $\conv (P)$. In particular, there is a non-recurrent point there, so the set of periodic points of $f$ is not connected. At the same time, by Theorem~\ref{T:hor-ent} the polynomial entropy of $f$ is not zero.
\end{proof}

\begin{thm}[Possible values of polynomial entropy of interval maps]\label{T:values}\strut
\begin{enumerate}
	\item For an interval map $f$ of Sharkovskii type 1, $\hpol(f) \in \mathbb N_0 \cup \{\infty\}$ and all these values are possible.
	\item For an interval map $f$ of Sharkovskii type $2^n$ where $n\geq 1$, we have $\hpol(f) \in \{n-1, n, n+1, \dots\}\cup \{\infty\}$ and all these values are possible.
	\item For an interval map $f$ of Sharkovskii type $2^{\infty}$ or greater, $\hpol (f) =\infty$.
\end{enumerate}	
\end{thm}

In order to prove this theorem we will use a standard period-doubling construction. Let $L=[0,\frac13]$, $M=[\frac13,\frac23]$, and $R=[\frac23,1]$. Consider a doubling operator $f \mapsto \Phi f$ which acts on interval maps as follows: $(\Phi f)(x)=\tfrac23 + \tfrac13 f(3x)$ for $x\in L$, $(\Phi f)(x)=x-\tfrac23$ for $x\in R$, and $(\Phi f)$ is the linear map on $M$ which connects the values which were already defined at the points $\tfrac13, \tfrac23$, see Figure~\ref{F:double}.
%\begin{equation*}
%(\Phi f)(x) = \begin{cases}
%\tfrac23 + \tfrac13 f(3x), & \text{ if } x\in L\\
%(2+f(1))(\tfrac23-x), & \text{ if } x\in M\\
%x-\tfrac23, & \text{ if } x\in R
%\end{cases}.,
%\end{equation*}
%see Figure~\ref{F:double}.
It is easy to check that $(\Phi f)(L) \subseteq R$, $(\Phi f)(R) \subseteq L$, there is a unique fixed point $c$ of $\Phi f$ and it lies in $M$, every other point of $M$ has a trajectory which eventually falls into $L\cup R$, and the restriction of the second iterate $(\Phi f)^2$ to $L$ is conjugate to $f$ by a linear rescaling, as shown in the following commutative diagram,
\begin{equation}\label{cd}
\begin{CD}
L @>(\Phi f)^2>> L\\
@V\times 3VV @VV\times 3V\\
[0,1] @>>f> [0,1]
\end{CD}.
\end{equation}

\begin{figure}[htb!!]
\begin{tikzpicture}
\draw (0,0) rectangle (3,3);
\node [label={above:$f$}] at (1.5,3) {};
\draw [smooth, samples=100,domain=0:3,thick] plot(\x,{0.8*sin(deg(3-3*\x))+0.2*(3-\x)*(3-\x)+0.57});
\begin{scope}[shift={(5,0)}]
\draw (0,0) grid (3,3);
\node [label={above:$\Phi f$}] at (1.5,3) {};
\end{scope}
\begin{scope}[scale=1/3,shift={(15,6)}]
\draw [smooth, samples=100,domain=0:3,thick] plot(\x,{0.8*sin(deg(3-3*\x))+0.2*(3-\x)*(3-\x)+0.57});
\draw [thick] (3,{0.8*sin(deg(-6))+0.57}) -- (6,-6) -- (9,-3);
\end{scope}
\end{tikzpicture}
\caption{The doubling operator $\Phi$ applied to an interval map.}\label{F:double}
\end{figure}
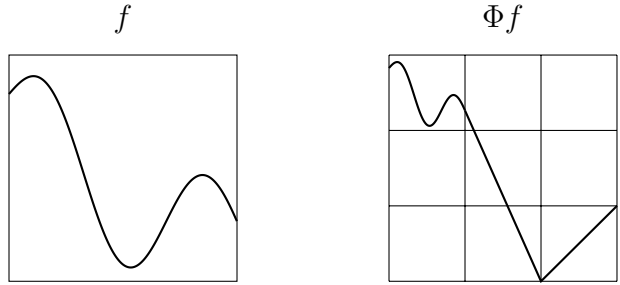

We continue to denote by $\Type(f)$ the Sharkovskii type of an interval map. Then the key properties of the doubling operator $\Phi$ are expressed in the following lemma (where we take $2\cdot 2^\infty = 2^\infty$).

\begin{lem}\label{L:doubling}
For any interval map $f$ we have $\Type(\Phi f)=2 \Type(f)$ and $\hpol(\Phi f) = \hpol(f) + 1$.
\end{lem}
\begin{proof}
If we let $P(\cdot)$ denote the set of all periods of all cycles of an interval map, then it is easy to see from the definition of the doubling operator that $P(\Phi f) = 2 P(f) \cup \{1\}$. In light of the Sharkovskii ordering, this immediately implies that $\Type(\Phi f) = 2 \Type(f)$.
%Suppose $x$ is periodic for $f$ with period $n$. Then $\frac{x}{3}$ is periodic for $\Phi f$ with period $2n$. Conversely, suppose $x$ is periodic for $\Phi f$. If $x\in L$, then since $\Phi f$ exchanges $L$ with $R$ the period of $x$ must be an even number. By~\eqref{cd} it follows that $3x$ is periodic for $f$ with period $n$. If $x\in R$, then a similar argument shows that $3(x-\frac23)$ is periodic for $f$ with period $n$. If $x$ is not in $L\cup R$, then it must be the unique fixed point of $\Phi f$ in $M$, since all other points in $M$ eventually fall into the invariant set $L\cup R$. 

To show the effect of $\Phi$ on the polynomial entropy we use Theorem~\ref{T:rigid} and consider one-way horseshoes. Let $c\in M$ be the unique fixed point of $\Phi f$ and let $b\in(\frac13,c)$ be the point with $(\Phi f)^2(b)=\frac13$. Put $B_0=[b,c]$. From the definition of $\Phi$ it is easy to see that 
\begin{equation}\label{B0}
(\Phi f)^m(B_0) = [0,c] \text{ for all even }m\geq 4,
\end{equation}
and clearly $B_0$ contains a non-recurrent point. We have constructed a one-way horseshoe of order 1 without using any properties of $f$. Now suppose that there are pairwise disjoint compact sets $A_1, \ldots, A_\ell$, $\ell\geq1$, forming a one-way horseshoe for some iterate $f^n$. By Lemma~\ref{L:horse-iter} we may assume that $n\geq 2$.
Put $B_1=\frac13 A_1$, \ldots, $B_\ell = \frac13 A_\ell$. These are pairwise disjoint compact sets contained in $L$, and each is disjoint from $B_0$ as well. Since the set $A_{\ell}$ contains a non-recurrent point of $f^n$, by~\eqref{cd} the set $B_{\ell}$ contains a non-recurrent point of $(\Phi f)^{2n}$. Using the inequality $2n\geq 4$, \eqref{cd} and~\eqref{B0} we see that the family $B_0, \ldots, B_\ell$ forms a one-way $(\ell+1)$-horseshoe for $(\Phi f)^{2n}$. In light of Theorem~\ref{T:rigid} this gives the inequality
\begin{equation}\label{Phi-gt}
\hpol(\Phi f) \geq \hpol(f) + 1.
\end{equation}

Conversely, let $A_1, \ldots, A_\ell$ be pairwise disjoint closed intervals forming a one-way horseshoe for some iterate $(\Phi f)^n$, and suppose $\ell\geq 2$. If some $A_i$ meets the interior of $M$, then since $(\Phi f)^n(A_1) \supseteq A_i$ we must have $A_1\cap\Int M\neq\emptyset$ as well, because $L\cup R$ is invariant for $(\Phi f)$. Additionally, since all points in $M$ move ``away'' from the fixed point $c$ toward $L\cup R$, the containment $(\Phi f)^n(A_1) \supseteq A_1$ implies that $c\in A_1$. The same argument shows that $c\in A_i$ for each $A_i$ which meets $\Int M$, so by pairwise disjointness, $A_1$ is the only member of the one-way horseshoe which can meet $\Int M$. In particular, we may assume $A_2 \subseteq L$ (the case $A_2 \subseteq R$ is similar). Now $(\Phi f)^m(A_2)$ is contained in $R$ for odd $m$ and in $L$ for even $m$. Since $(\Phi f)^n(A_2) \supseteq A_2$, $n$ must be even. Then since $(\Phi f)^n(A_2) \supseteq A_3, \ldots, A_\ell$, we see that all of the sets $A_2, \ldots, A_\ell$ are contained in $L$. Let $B_i = 3 A_i$ for $i=2,\ldots,\ell$. These are pairwise disjoint closed intervals in $[0,1]$, and by \eqref{cd} they form a one-way horseshoe for $f^{2n}$. Starting with a one-way $\ell$-horseshoe, $\ell\geq2$, for an iterate of $(\Phi f)$, we constructed a one-way $(\ell-1)$-horseshoe for an iterate of $f$. In light of theorem~\ref{T:rigid} this shows that
\begin{equation}\label{Phi-lt}
\hpol(f) \geq \hpol(\Phi f) - 1.
\end{equation}
The two inequalities~\eqref{Phi-gt},~\eqref{Phi-lt} together show that $\hpol(\Phi f)=\hpol(f)+1$, as desired. \qedhere
\end{proof}

\begin{proof}[Proof of Theorem~\ref{T:values}]
We know by Theorem~\ref{T:rigid} that the polynomial entropy is limited to integer values.

(1) Let $g_0$ be the identity map on $[0,1]$. Clearly $g_0$ has Sharkovskii type 1 and $\hpol(g_0)=0$. Now for $n\in\mathbb{N}$ let $g_n$ be a map with $\Fix(g_n)=\{0,\frac1n,\frac2n,\ldots,1\}$ such that $g(x)>x$ for $x\not\in\Fix(g)$ and
\begin{equation*}
g_n\left(\left[\tfrac{i-1}{n},\tfrac{i}{n}\right]\right)=\left[\tfrac{i-1}{n},\tfrac{\min\{n,i+1\}}{n}\right], \qquad i=1,\ldots,n,
\end{equation*}
see the top row of Figure~\ref{F:gs}. Clearly each $g_n$ has Sharkovskii type 1 and has exactly $n$ essential intervals, all arranged in one long chain. By Theorem~\ref{T:type1} we have $\hpol(g_n)=n$. Finally, let $g_{\infty}$ be an interval map with $\Fix(g_\infty)=\{0,\frac12,\frac23,\frac34,\ldots,1\}$ such that $g(x)>x$ for $x\not\in\Fix(g)$ and 
\begin{equation*}
g_\infty\left(\left[\tfrac{i-1}{i},\tfrac{i}{i+1}\right]\right)=\left[\tfrac{i-1}{i},\tfrac{i+1}{i+2}\right], \qquad i=1,2,\ldots,
\end{equation*}
see the top right map in Figure~\ref{F:gs}. Clearly $g_\infty$ has Sharkovskii type 1 and has arbitrarily long chains of essential intervals. By Theorem~\ref{T:type1} we have $\hpol(g_\infty)=\infty$.

\newcommand\doubler[2]{
	\begin{scope}[scale=2/3]
	\draw (0,0) -- (0,2) (1,0) -- (1,2) (2,1) -- (2,3) (3,1) -- (3,3);
	\draw (0,0) -- (2,0) (0,1) -- (2,1) (1,2) -- (3,2) (1,3) -- (3,3);
	\begin{scope}[scale=1/2,shift={(0,4)}] #1 \end{scope}
	\draw[thick] (1,#2+2) -- (2,0);
	\begin{scope}[scale=1/2,shift={(4,0)}]
	\draw (0,0) rectangle (2,2);
	\draw[thick] (0,0) -- (2,2); \end{scope}
	\end{scope}
}

\newcommand\drawgzero{
\draw[fill=white!90!black] (0,0) rectangle (2,2);
\draw[thick] (0,0) -- (2,2);
}

\newcommand\drawgone{
\draw[fill=white!90!black] (0,0) rectangle (2,2);
\draw (0,0) -- (2,2);
\draw[thick] (0,0) -- (1,2) -- (2,2);
}

\newcommand\drawgtwo{
\draw[fill=white!90!black] (0,0) rectangle (2,2);
\draw (0,1) -- (2,1) (1,0) -- (1,2);
\draw (0,0) -- (2,2);
\draw[thick, line cap=butt] (0,0) -- (0.5,2) -- (1,1) -- (1.5,2) -- (2,2);
}

\newcommand\drawgthree{
\draw[fill=white!90!black] (0,0) rectangle (2,2);
\draw (2/3,0) -- (2/3,2) (4/3,0) -- (4/3,2) (0,2/3) -- (2,2/3) (0,4/3) -- (2,4/3);
\draw (0,0) -- (2,2);
\draw[thick, line cap=butt] (0,0) -- (1/3,4/3) -- (2/3,2/3) -- (3/3,6/3) -- (4/3,4/3) -- (5/3,6/3) -- (6/3,6/3);
}

\newcommand\drawgminusone{
\draw[fill=white!90!black] (0,0) rectangle (2,2);
\draw[thick] (0,2) -- (2,0);
}

\newcommand\drawginfty{
\draw[fill=white!90!black] (0,0) rectangle (2,2);
\foreach \i in {1,...,3} {
	\draw ({2*\i/(\i+1)},0) -- ({2*\i/(\i+1)},2);
	\draw (0,{2*\i/(\i+1)}) -- (2,{2*\i/(\i+1)});
}
\foreach \t in {1.7,1.8,1.9} {
	\draw[fill] (\t,0.25) circle [radius=0.015];
	\draw[fill] (0.25,\t) circle [radius=0.015];
}
\foreach \i in {0,...,100} {
	\draw[thick] ({2*\i/(\i+1)},{2*\i/(\i+1)}) -- ({\i/(\i+1)+(\i+1)/(\i+2)},{2*(\i+2)/(\i+3)}) -- ({2*(\i+1)/(\i+2)},{2*(\i+1)/(\i+2)});
}
\draw (0,0) -- (2,2);
}

\begin{figure}[htb!!]
\begin{tikzpicture}

\drawgzero \node[above] at (1,2.5) {$\hpol=0$}; \node[left] at (-0.5,1) {Type $2^0$};
\begin{scope}[shift={(2.5,0)}]	\drawgone \node[above] at (1,2.5) {$\hpol=1$}; \end{scope}
\begin{scope}[shift={(5,0)}] \drawgtwo \node[above] at (1,2.5) {$\hpol=2$}; \end{scope}
\begin{scope}[shift={(7.5,0)}] \drawgthree \node[above] at (1,2.5) {$\hpol=3$}; \end{scope}
\begin{scope}[shift={(11,0)}] \drawginfty \node[above] at (1,2.5) {$\hpol=\infty$}; \end{scope}
\begin{scope}[shift={(0,-2.5)}] \drawgminusone \node[left] at (-0.5,1) {Type $2^1$}; \end{scope}
\begin{scope}[shift={(2.5,-2.5)}] \doubler{\drawgzero}{1} \end{scope}
\begin{scope}[shift={(5,-2.5)}] \doubler{\drawgone}{1} \end{scope}
\begin{scope}[shift={(7.5,-2.5)}] \doubler{\drawgtwo}{1} \end{scope}
\begin{scope}[shift={(11,-2.5)}] \doubler{\drawginfty}{1} \end{scope}
\begin{scope}[shift={(0,-5)}] \node at (1,1) {$\boldsymbol{\times}$}; \end{scope}
\begin{scope}[shift={(2.5,-5)}] \doubler{\drawgminusone}{0} \node[left] at (-3,1) {Type $2^2$}; \end{scope}
\begin{scope}[shift={(5,-5)}] \doubler{\doubler{\drawgzero}{1}}{1/3} \end{scope}
\begin{scope}[shift={(7.5,-5)}] \doubler{\doubler{\drawgone}{1}}{1/3} \end{scope}
\begin{scope}[shift={(11,-5)}] \doubler{\doubler{\drawginfty}{1}}{1/3} \end{scope}
\begin{scope}[shift={(0,-7.5)}] \node at (1,1) {$\boldsymbol{\times}$}; \end{scope}
\begin{scope}[shift={(2.5,-7.5)}] \node at (1,1) {$\boldsymbol{\times}$}; \end{scope}
\begin{scope}[shift={(5,-7.5)}] \doubler{\doubler{\drawgminusone}{0}}{1/3} \node[left] at (-5.5,1) {Type $2^3$}; \end{scope}
\begin{scope}[shift={(7.5,-7.5)}] \doubler{\doubler{\doubler{\drawgzero}{1}}{1/3}}{1/3} \end{scope}
\begin{scope}[shift={(11,-7.5)}] \doubler{\doubler{\doubler{\drawginfty}{1}}{1/3}}{1/3} \end{scope}
\foreach \y in {1,-1.5,-4,-6.5}{ \node at (10.25,\y) {$\cdots$}; }
\foreach \x in {2.5,5,7.5}{ \draw[thick,->] (\x-0.4,-.1) -- (\x-.1,-0.4);}
\foreach \x in {2.5,5,7.5}{ \draw[thick,->] (\x-0.4,-2.6) -- (\x-.1,-2.9);}
\foreach \x in {5,7.5}{ \draw[thick,->] (\x-0.4,-5.1) -- (\x-.1,-5.4);}
\foreach \y in {0,-2.5,-5} { \draw[thick,->] (12,\y-0.1) -- (12,\y-0.4); }
\end{tikzpicture}

\caption{Possible combinations of polynomial entropy and Sharkovskii type.}\label{F:gs}
\end{figure}

(2) Suppose $f$ has Sharkovskii type $2^n$, $n\geq 1$. If $n=1$ then the inequality $\hpol(f)\geq n-1$ follows for free. If $n\geq2$ then we use the fact that $f$ has a cycle $P$ of period $2^n$, so it has a simple cycle of the same period. Using Proposition~\ref{P:simple-horse} we again have $\hpol(f)\geq n-1$. Since the polynomial entropy is limited to integer values we have shown that $\hpol(f)\in\{n-1,n,n+1,\ldots\}\cup\{\infty\}$, and it remains to show that any value in this set is possible.

The map $f_0(x)=1-x$ has Sharkovskii type 2 and polynomial entropy zero, since its second iterate is the identity, see the first map in row 2 of Figure~\ref{F:gs}. Now we can use the doubling operator (as indicated by the arrows in Figure~\ref{F:gs}) to find maps with the remaining combinations of polynomial entropy and Sharkovskii type. Using Lemma~\ref{L:doubling} we see that for $n\geq1$,
\begin{equation*}
\Type(\Phi^{n-1} f_0) = 2^n \quad \text{and} \quad \hpol(\Phi^{n-1}f_0) = n-1.
\end{equation*}
Using the maps $g_m$ from part (1) of this proof and Lemma~\ref{L:doubling} we see that for $n\geq 1$ and $m\in\Nzero\cup\{\infty\}$,
\begin{equation*}
\Type(\Phi^n g_m) = 2^n \quad \text{and} \quad \hpol(\Phi^n g_m) = n+m.
\end{equation*}
Together this shows that each of the numbers in $\{n-1,n,n+1,\ldots\}\cup\{\infty\}$ can be attained as the polynomial entropy of an interval map of Sharkovskii type $2^n$.

(3) Suppose $f$ has Sharkovskii type $2^\infty$ or higher. Then it has cycles of period $2^n$ for all $n$. In particular, it has simple cycles of period $2^n$ for all $n$, so by Proposition~\ref{P:simple-horse} we have $\hpol(f)=\infty$.
\end{proof}

%%%%%%%%%%%%%%%%%%%%%%%%%%%%%%%%%%%%%%%%%

\subsection{Application to the logistic family}\label{SS:logistic}	

%%%%%%%%%%%%%%%%%%%%%%%%%%%%%%%%%%%%%%%%%
Let $f_\lambda(x)=\lambda x(1-x)$ be the usual logistic family of interval maps $f_\lambda:[0,1]\to[0,1]$ with parameter $\lambda\in[0,4]$. We calculate the polynomial entropy $h(f_\lambda)$ for parameter values within the period-doubling cascade, up to $\lambda_{\infty}$ where the first solenoid appears and the set of periodic points is no longer finite. We freely use known facts about the logistic family, such as the order of appearance and the combinatorial arrangement in $[0,1]$ of the cycles of $f_\lambda$ during the period-doubling cascade. For a general reference, see~\cite{SMR93}, \cite{Str15}, or \cite{MT88}.
Some preliminary cases are quite easy at this point. For $0\leq\lambda\leq1$ the map $f_\lambda$ has Sharkovskii type 1 and a unique fixed point at zero, so by Corollary~\ref{C:zero} the polynomial entropy is zero. For $1<\lambda\leq3$ the map $f_\lambda$ still has Sharkovskii type 1, but there are two distinct fixed points. Therefore there is a unique essential interval and every chain of essential intervals must have length 1. So by Theorem~\ref{T:type1} the polynomial entropy is one.

Now fix $\lambda$ with $3<\lambda<\lambda_{\infty}$. Then there is $n\geq1$ such that $f_\lambda$ has an attracting periodic cycle $P$ of period $2^n$. This cycle is only ``weakly attracting'' if $\lambda$ is a bifurcation parameter (the largest value of $\lambda$ when $f_{\lambda}$ is of Sharkovskii type $2^n$), i.e.\ the product of the derivatives around the cycle has absolute value 1 in this case, but $P$ is still a topological attractor so it does not bother us. We can enumerate $P=\{x_w\}_{w\in\{0,1\}^n}$ in the spatial ordering so that $x_v < x_w$ in the interval if and only if $v<w$ in the lexicographical ordering. For each $m<n$ there is a repelling $2^m$-cycle $Q_m=\{y_w\}_{w\in\{0,1\}^m}$, again labelled in the spatial ordering (for $m=0$ we write simply $Q_0=\{y\}$), as well as an additional repelling fixed point at $0$, for a total of $1+\sum_{m<n} 2^m + 2^n = 2^{n+1}$ periodic points, see Figure~\ref{F:tree}. Passing to the iterate $g=f_{\lambda}^{2^n}$, all of these $2^{n+1}$ periodic points become fixed points, and since $\Per(g)=\Per(f_\lambda)$ we see that $g$ is a map of Sharkovskii type 1. Thus we may apply Theorem~\ref{T:type1}. By the power rule, $\hpol(f_\lambda)=\hpol(g)$, so we may calculate the polynomial entropy of $f_\lambda$ by finding the chains of essential intervals of $g$.	
We label the essential intervals of $g$ as follows, where the underlined words have enough 0's or 1's at the end to reach total length $n$,	
\begin{align*}	
I &= (0, x_{\underline{00\dots0}}) \\	
J_w^0 &= (x_{\underline{w011\dots1}}, y_w), \qquad \qquad w \in \smash{\bigcup_{m=0}^{n-1} \{0,1\}^m,} \\	
J_w^1 &= (y_w, x_{\underline{w100\dots0}})	
\end{align*}
(see Figure~\ref{F:tree} for $n=3$). Then $I$ and each $J_w^1$ are up intervals (since the left endpoint is repelling and the right endpoint attracting), while each $J_w^0$ is a down interval. We need to calculate which of them contain which others in their $g$-orbits. Clearly each contains itself. We will use \emph{arrow notation} $A \longrightarrow B$ to indicate that $\Orb_g(A) \supseteq B$ whenever $A, B$ are distinct essential intervals.	
For $I$ we have the left endpoint fixed at 0 and the right endpoint the smallest element of $P$, so $f_\lambda(I)\supseteq I$. This gives a chain of inclusions	
\begin{equation*}	
\underbracket[.187ex][.35ex]{I \subseteq f_\lambda(I) \subseteq f^2_\lambda(I) \subseteq \dots}_{2^n\text{ terms}} \subseteq f^{2^n}_\lambda(I) = g(I),	
\end{equation*}	
in which the full orbit $P$ appears in the first $2^n$ terms (possibly taking closures). Thus $g(I)$ contains all the essential intervals contained in $\conv(P)$, that is, all the $J_w^i$'s. In particular	
\begin{equation}\label{arrow1}	
I \longrightarrow J^0 \text{ and } I \longrightarrow J^1.	
\end{equation}	
Now fix $0\leq m \leq n-2$ and $w\in\{0,1\}^m$ and consider the up interval $J_w^1$. We work with the map $f_\lambda^k$ where $k=2^{m+1}$. Then the left endpoint $y_w$ is fixed while the right endpoint is the smallest endpoint of a $2^{n-1-m}$-cycle $P_{w1} = \{x_{w1v}\}_{v\in\{0,1\}^{n-1-m}}$. This shows that $f^k_\lambda(J^1_w) \supseteq J^1_w$, so we again get a chain of inclusions	
\begin{equation*}	
\underbracket[.187ex][.35ex]{J^1_w \subseteq f^k_\lambda(J^1_w) \subseteq f^{2k}_\lambda(J^1_w) \subseteq \dots}_{2^{n-1-m}\text{ terms}} \subseteq f^{2^n}_\lambda(J^1_w) = g(J^1_w),	
\end{equation*}	
in which the full orbit $P_{w1}$ appears in the first $2^{n-1-m}$ terms (possibly taking closures). Thus $g(J^1_w)$ contains all the essential intervals in $\conv(P_{w1})$. In particular there are arrows pointing from the up interval $J^1_w$ to $J^0_{w1}$ and to $J^1_{w1}$. Applying the same kind of argument to the down interval $J^0_w$ we conclude in both cases that there are arrows	
\begin{equation}\label{arrow2}	
J^i_w \longrightarrow J^j_{wi}, \qquad i,j\in\{0,1\},\ w\in\{0,1\}^m,\ 0\leq m\leq n-2.	
\end{equation}	
The arrows identified in~\eqref{arrow1} and~\eqref{arrow2} place the essential intervals of $g$ into a binary tree structure in which there are chains of essential intervals of the form
\begin{equation}\label{arrow3}
I \longrightarrow J^{w_0} \longrightarrow J^{w_1}_{w_0} \longrightarrow J^{w_2}_{w_0w_1} \longrightarrow \dots \longrightarrow J^{w_{n-1}}_{w_0 \dots w_{n-2}},	
\end{equation}
for all $w\in\{0,1\}^n$, see Figure~\ref{F:tree}. 

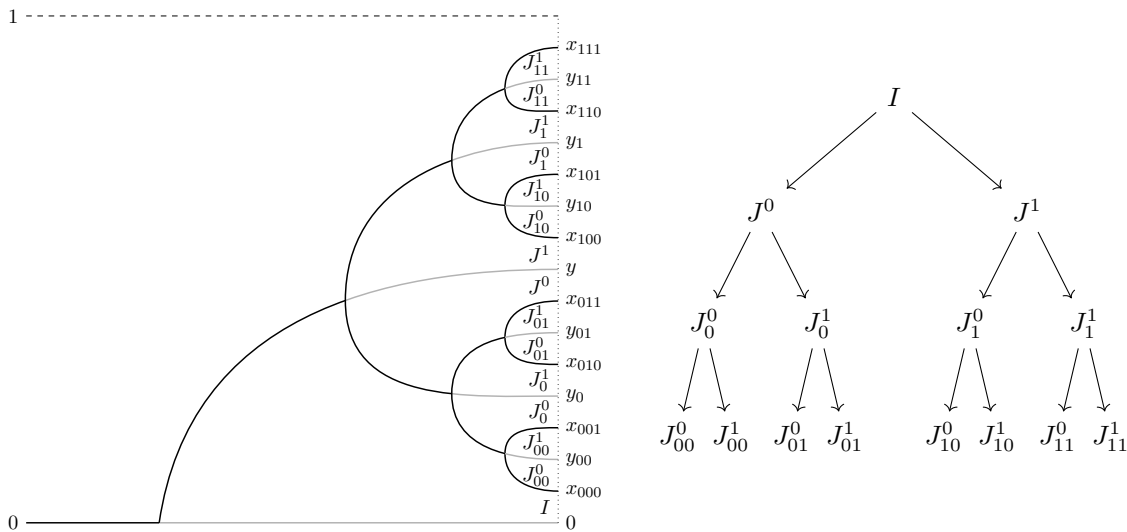
\begin{figure}[htb!!]
\raisebox{-0.5\height}{\scalebox{.7}{
\begin{tikzpicture}[xscale=.5,yscale=.6]
\path [name path=P1] (0,8) to[out=180,in=80] (-15,0);
\path [name path=minus8] (-8,0) -- (-8,16);
\path [name intersections={of=P1 and minus8,by=bif1}];
\path [name path=P2] (0,4) to[out=180,in=270] (bif1) to[out=90,in=180] (0,12);
\path [name path=minus4] (-4,0) -- (-4,16);
\path [name intersections={of=P2 and minus4,by={bif2low, bif2high}}];
\path [name path=P3] (0,2) to[out=180,in=270] (bif2low) to[out=90,in=180] (0,6);
\path [name path=P4] (0,10) to[out=180,in=270] (bif2high) to[out=90,in=180] (0,14);
\path [name path=minus2] (-2,0) -- (-2,16);
\path [name intersections={of=P3 and minus2,by={bif3low, bif3high}}];
\path [name intersections={of=P4 and minus2,by={bif4low, bif4high}}];
\draw [thick] (0,1) to[out=180,in=270] (bif3low) to[out=90,in=180] (0,3);
\draw [thick] (0,5) to[out=180,in=270] (bif3high) to[out=90,in=180] (0,7);
\draw [thick] (0,9) to[out=180,in=270] (bif4low) to[out=90,in=180] (0,11);
\draw [thick] (0,13) to[out=180,in=270] (bif4high) to[out=90,in=180] (0,15);
\draw [thick] (-20,0) -- (-15,0);
\draw [thick, white!70!black] (-15,0) -- (0,0);

\begin{scope}\clip(-8,0) rectangle (0,16);\draw [thick, white!70!black] (0,8) to[out=180,in=80] (-15,0);\end{scope}
\begin{scope}\clip(-16,0) rectangle (-8,16);\draw [thick] (0,8) to[out=180,in=80] (-15,0);\end{scope}
\begin{scope}\clip(-4,0) rectangle (0,16);\draw [thick, white!70!black] (0,4) to[out=180,in=270] (bif1) to[out=90,in=180] (0,12);\end{scope}
\begin{scope}\clip(-16,0) rectangle (-4,16);\draw [thick] (0,4) to[out=180,in=270] (bif1) to[out=90,in=180] (0,12);\end{scope}
\begin{scope}\clip(-2,0) rectangle (0,16);\draw [thick, white!70!black] (0,2) to[out=180,in=270] (bif2low) to[out=90,in=180] (0,6);\end{scope}
\begin{scope}\clip(-16,0) rectangle (-2,16);\draw [thick] (0,2) to[out=180,in=270] (bif2low) to[out=90,in=180] (0,6);\end{scope}
\begin{scope}\clip(-2,0) rectangle (0,16);\draw [thick, white!70!black] (0,10) to[out=180,in=270] (bif2high) to[out=90,in=180] (0,14);\end{scope}
\begin{scope}\clip(-16,0) rectangle (-2,16);\draw [thick] (0,10) to[out=180,in=270] (bif2high) to[out=90,in=180] (0,14);\end{scope}

\draw [dotted] (0,0) -- (0,16);
\draw [dashed] (-20,16) -- (0,16);
\node [left] at (-20,16) {\small $1$};
\node [left] at (-20,0) {\small $0$};
\node [left] at (0,.5) {\small $I$};
\node [left] at (0,1.5) {\small $J^0_{00}$};
\node [left] at (0,2.5) {\small $J^1_{00}$};
\node [left] at (0,3.5) {\small $J^0_0$};
\node [left] at (0,4.5) {\small $J^1_0$};
\node [left] at (0,5.5) {\small $J^0_{01}$};
\node [left] at (0,6.5) {\small $J^1_{01}$};
\node [left] at (0,7.5) {\small $J^0$};
\node [left] at (0,8.5) {\small $J^1$};
\node [left] at (0,9.5) {\small $J^0_{10}$};
\node [left] at (0,10.5) {\small $J^1_{10}$};
\node [left] at (0,11.5) {\small $J^0_1$};
\node [left] at (0,12.5) {\small $J^1_1$};
\node [left] at (0,13.5) {\small $J^0_{11}$};
\node [left] at (0,14.5) {\small $J^1_{11}$};
\node [right] at (0,0) {\small $0$};
\node [right] at (0,1) {\small $x_{000}$};
\node [right] at (0,3) {\small $x_{001}$};
\node [right] at (0,5) {\small $x_{010}$};
\node [right] at (0,7) {\small $x_{011}$};
\node [right] at (0,9) {\small $x_{100}$};
\node [right] at (0,11) {\small $x_{101}$};
\node [right] at (0,13) {\small $x_{110}$};
\node [right] at (0,15) {\small $x_{111}$};
\node [right] at (0,2) {\small $y_{00}$};
\node [right] at (0,6) {\small $y_{01}$};
\node [right] at (0,10) {\small $y_{10}$};
\node [right] at (0,14) {\small $y_{11}$};
\node [right] at (0,4) {\small $y_0$};
\node [right] at (0,12) {\small $y_1$};
\node [right] at (0,8) {\small $y$};
\end{tikzpicture}
}}
\tikzset{
  font={\fontsize{10pt}{12}\selectfont}}
\raisebox{-.5\height}{
\begin{tikzpicture}
[->,level distance=1.5cm,
  level 1/.style={sibling distance=3.5cm},
  level 2/.style={sibling distance=1.5cm},
  level 3/.style={sibling distance=0.7cm}]
  \node {$I$}
    child {node {$J^0$}
      child {node {$J^0_0$}
      	child {node {$J^0_{00}$}}
		child {node {$J^1_{00}$}}
      }
      child {node {$J^1_0$}
      	child {node {$J^0_{01}$}}
		child {node {$J^1_{01}$}}
      }
    }
    child {node {$J^1$}
	  child {node {$J^0_1$}
	  	child {node {$J^0_{10}$}}
		child {node {$J^1_{10}$}}
	  }
      child {node {$J^1_1$}
      	child {node {$J^0_{11}$}}
		child {node {$J^1_{11}$}}
      }
    };
\end{tikzpicture}
}
\caption{Chains of essential intervals for the logistic map $f_\lambda$ when there is an attracting 8-cycle (i.e.\ $n=3$). The bifurcation diagram is distorted, but shows the correct arrangement and order of appearance of the periodic orbits, with attractors in black and repellors in gray. Each period doubling bifurcation produces new essential intervals and adds another level to the binary tree.}
\label{F:tree}	
\end{figure}

We call this tree $T$ and we wish to verify that there are no additional arrows except those that follow from transitivity, i.e.\ that an essential interval's descendants in $T$ are the only essential intervals contained in its $g$-orbit. In the examples in Figures~\ref{F:f2andf4}(\subref{F:f2}) and~\ref{F:f2andf4}(\subref{F:f4}) this can well be seen and we explain why this is always true.

\begin{figure}
\begin{subfigure}[b]{.35\textwidth}
\begin{center}
\begin{tikzpicture}
	%% 2nd iterate of a logistic map of type 2
	[scale=4,
	declare function={f(\t)=3.449*\t*(1-\t);},
	declare function={f2(\t)=f(f(\t));}]
	%% square and diagonal
	\draw [thin] (0,0) rectangle (1,1);
	\draw [thin] (0,0)--(1,1);
	%% graphs
	%\draw [very thin] plot [domain=0:1, samples=100, smooth] (\x,{f(\x)}); 
	\draw [thick] plot [domain=0:1, samples=500, smooth] (\x,{f2(\x)});
	%% nodes
	\draw [dashed, thin] (0.440075, -0.02) -- (0.440075, 0.440075);
	\node at (0.22,-0.06) {$I$};
	\draw [dashed, thin] (0.710061, -0.02) -- (0.710061, 0.710061);
	\node at (0.575,-0.13) {$J^0$};
	\draw [dashed, thin] (0.849864, -0.02) -- (0.849864, 0.849864);
	\node at (0.78,-0.06) {$J^1$};
\end{tikzpicture}\\[3em]
\captionsetup{justification=centering}
\caption{$g=f_{\lambda}^2$ if $\Type(f_{\lambda})=2$;\linebreak here $\lambda = 3.449$}
\label{F:f2}
\end{center}
\end{subfigure}
\begin{subfigure}[b]{.55\textwidth}
\begin{center}
\begin{tikzpicture}
	%% 4th iterate of a logistic map of type 4
	[scale=6,
	declare function={f(\t)=3.544*\t*(1-\t);},
	declare function={f4(\t)=f(f(f(f(\t))));}]
	%%square and diagonal
	\draw [thin] (0,0) rectangle (1,1);
	\draw [thin] (0,0)--(1,1);
	%% graphs
    %\draw [very thin] plot [domain=0:1, samples=100, smooth] (\x,{f(\x)}); 
	\draw [thick] plot [domain=0:1, samples=2000, smooth] (\x,{f4(\x)}); 
	%% nodes
	\draw [dashed, thin] (0.363324, -0.02) -- (0.363324, 0.363324);
	\node at (0.18,-0.06) {$I$};
	\draw [dashed, thin] (0.419266, -0.02) -- (0.419266, 0.419266);
	\node at (0.39,-0.13) {$J^0_0$};
	\draw [dashed, thin] (0.523554, -0.02) -- (0.523554, 0.523554);
	\node at (0.47,-0.06) {$J^1_0$};
	\draw [dashed, thin] (0.717833, -0.02) -- (0.717833, 0.717833);
	\node at (0.62,-0.13) {$J^0$};
	\draw [dashed, thin] (0.819797, -0.02) -- (0.819797, 0.819797);
	\node at (0.77,-0.06) {$J^1$};
	\draw [dashed, thin] (0.862901, -0.02) -- (0.862901, 0.862901);
	\node at (0.84,-0.13) {$J^0_1$};
	\draw [dashed, thin] (0.884033, -0.02) -- (0.884033, 0.884033);
	\node at (0.88,-0.06) {$J^1_1$};
\end{tikzpicture} 
\captionsetup{justification=centering}
\caption{$g=f_{\lambda}^4$ if $\Type(f_{\lambda})=4$;\linebreak here $\lambda = 3.544$}
\label{F:f4}
\end{center}
\end{subfigure}
\caption{Essential intervals for iterates of maps in the logistic family.}
\label{F:f2andf4}
\end{figure}
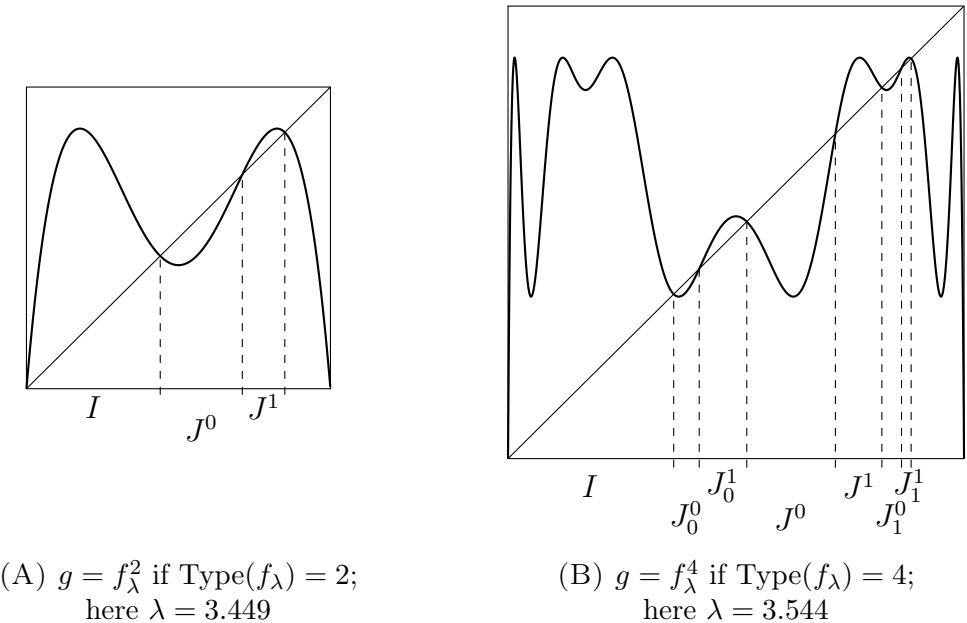

By Remark~\ref{r:no-loops} there are no nontrivial cycles of essential intervals, so no essential interval in our tree $T$ contains any of its ancestors in its $g$-orbit. But if we take an essential down interval $J^0_w$ and look to the left of it in $[0,1]$, the first essential interval we come to which is not one of its descendants in $T$ is one of its ancestors. The same is true when we look to the right of an essential up interval (unless all the essential intervals to the right of it in $[0,1]$ are already its descendants in $T$, as is the case for $I$). Thus by Lemma~\ref{lem:orb-ess} our tree already shows the full $g$-orbit of each essential interval. The chains of essential intervals identified in~\eqref{arrow3} are therefore maximal chains, and they all have length $n+1$. Applying Theorem~\ref{T:type1} we reach the following conclusion:

\begin{thm}\label{T:logistic}
Within the logistic family, $\hpol(f_\lambda) = 0$ when there is an attracting fixed point at zero, and for parameter values $1<\lambda<\lambda_{\infty}$ in the period-doubling cascade, $\hpol(f_\lambda) = n+1$ when there is an attracting $2^n$-cycle.
\end{thm}

The picture is completed by noting that $\hpol(f_\lambda)=\infty$ for $\lambda\geq\lambda_{\infty}$ (at the Feigenbaum point and beyond) since by Theorem~\ref{T:values} a map of Sharkovskii type $2^\infty$ or greater always has infinite polynomial entropy. Figure~\ref{fig:bif} shows polynomial entropy and Sharkovskii type overlayed on the bifurcation diagram for the logistic family.

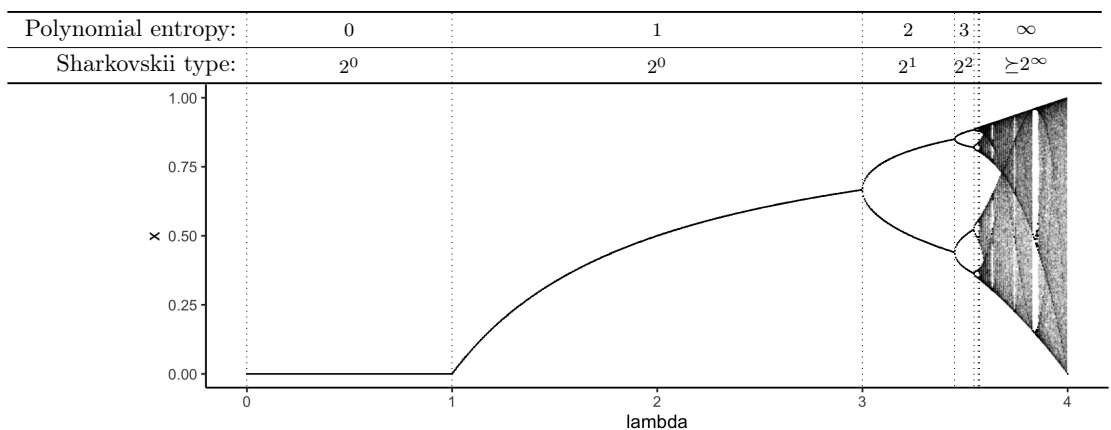
\begin{figure}[htb!!]
\strut\hfill
\scalebox{.9}{
\begin{tikzpicture}[yscale=1.05]
\node[inner sep=0pt] (diagram) at (5.57,1.6) {\includegraphics[width=14.33cm]{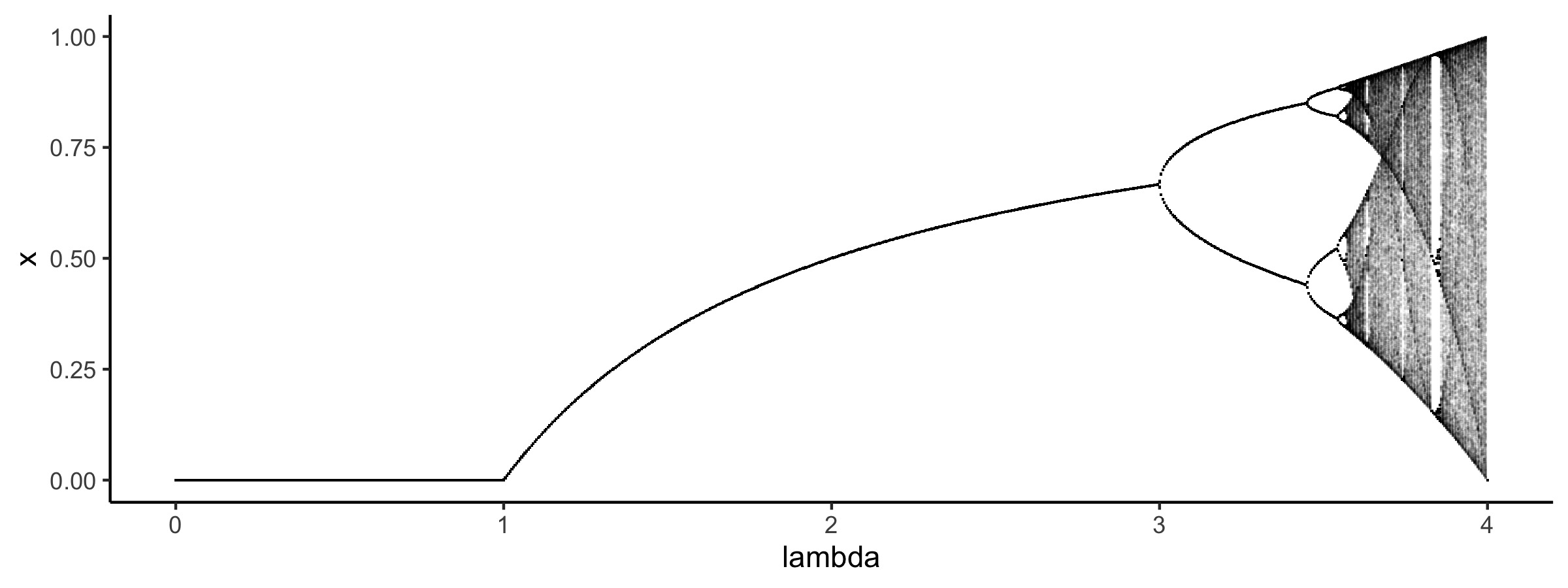}};
%\draw[blue] (0,0) grid (12,4); % domain lambda=0:4 in tikz is 0:12
\foreach \x in {0,1,3,3.449,3.5440903596,3.5644072661,3.5687594195,3.5696916098} {
	\draw[black,thin,dotted] (3*\x,-.2) -- (3*\x,5.05);
}
\draw[black,thick] (-3.5,5.05) -- (12.5,5.05);
\node[left] at (0,4.80) {\footnotesize Polynomial entropy:};
\draw[thin] (-3.5,4.55) -- (12.5,4.55);
\node[left] at (0,4.30) {\footnotesize Sharkovskii type:};
\draw[black,thick] (-3.5,4.05) -- (12.5,4.05);
\node at (1.5,4.3) {$\scriptstyle 2^0$};
\node at (1.5,4.8) {$\scriptstyle 0$};
\node at (6,4.3) {$\scriptstyle 2^0$};
\node at (6,4.8) {$\scriptstyle 1$};
\node at (9.6735,4.3) {$\scriptstyle 2^1$};
\node at (9.6735,4.8) {$\scriptstyle 2$};
\node at (10.4895,4.3) {$\scriptstyle 2^2$};
\node at (10.4895,4.8) {$\scriptstyle 3$};
\node at (11.4,4.3) {$\scriptstyle \succeq 2^\infty$};
\node at (11.4,4.8) {$\scriptstyle \infty$};
\node at (15,0) {\phantom{.}};
\end{tikzpicture}
}
\hfill\hfill\strut
\vspace{-1em}
\caption{Polynomial entropy in the logistic family.}\label{fig:bif}
\end{figure}

\begin{rem}
There is another important theorem due to Misiurewicz and Szlenk giving the topological entropy of a piecewise monotone interval map $f$ in terms of the exponential growth rate of the so-called \emph{lap numbers} $c_n$ (the number of maximal intervals of monotonicity of $f^n$) \cite{MS80}. It is then natural to ask if the polynomial growth rate of lap numbers gives the polynomial entropy. In~\cite{GC21} the following upper bound was given
\begin{equation}\label{upper}
\hpol(f) \leq 1 + \limsup_{n\to\infty} \frac{\log c_n}{\log n}.
\end{equation}
Within the logistic family, all iterates $f_\lambda^n$ have only $c_n=2$ laps for $\lambda\leq 2$ (up to the first moment when the critical point is periodic), giving $\hpol(f_\lambda)\leq1$ for these parameter values. In fact, a complete calculation is given in~\cite{GC21} showing that $\hpol(f_\lambda)=0$ for $\lambda\leq 1$ and $\hpol(f_\lambda)=1$ for $\lambda \in (1,2]$, thus showing that maps with the same lap numbers can have different polynomial entropies (and the ``plus 1'' term in~\eqref{upper} cannot be omitted).

Milnor and Thurston calculated lap numbers in the logistic family, showing that the lap numbers are $c_n=2n$ (linear growth) for all maps $f_\lambda$ with $\lambda\in(2,4+\sqrt5]$ (up to the next parameter value when the critical point is periodic), then $c_n=n^2-n+2$ (quadratic growth) until the next such parameter value, and so on with the degree increasing by one each time the critical point is periodic until the Feigenbaum point~\cite{MT88}. It is interesting that the parameter values when the critical point is periodic alternate with the period-doubling bifurcations, so by Theorem~\ref{T:logistic} the polynomial entropy is given by the polynomial growth rate of lap numbers plus either zero or one, depending on which of these types of parameter values occurred most recently.

Some open questions were stated at the end of~\cite{GC21}, asking whether the polynomial entropy $\hpol(f_\lambda)$ is a nondecreasing function of $\lambda$ and whether it takes values only in $\mathbb{N}_0\cup\{\infty\}$ within the logistic family. Our work gives affirmative answers to both of those questions. Note that the latter question anticipates our rigidity result, at least in the setting of the logistic family.

It was also asked in~\cite{GC21} whether the polynomial growth rate of lap numbers determines a lower bound for the polynomial entropy, and we feel that this question is important enough to repeat here.
\end{rem}

\begin{Q}
Does the inequality $\hpol(f) \geq \limsup_{n\to\infty} \frac{\log c_n}{\log n}$ hold for piecewise monotone interval maps? 
\end{Q}

\section{Flexibility of Polynomial Entropy}

\subsection{Flexibility for homeomorphisms on continua}\label{SS:HomeoCont}

We are going to prove that for homeomorphisms on continua, polynomial entropy can take arbitrary values in $[0,\infty]$.
We will build on the following fact.

\begin{prop}[{\cite[Proposition 3.5 and Remark 3.6]{ACM17}}]\label{P:denseA}
	There is a dense set $A\subseteq\mathbb{R}^+$ such that for every $a\in A$ there is a compact metric space $X_a$ and a homeomorphism $f_a:X_a\to X_a$ with $\hpol(f_a)=a$.
\end{prop}

Given a sequence of topological dynamical systems $(X_i, f_i)$, $i=1,2,\ldots$, we may make a disjoint topological sum of these systems as follows. We may assume that the spaces $X_i$ are pairwise disjoint, and we may rescale the metrics so that $X_i$ has diameter $1/2^{i-1}$ with respect to its metric $d_{X_i}$. Then put
\begin{equation*}
X=\{\infty\} \sqcup \bigsqcup_{i=1}^\infty X_i
\end{equation*}
with the following metric and dynamics: 
\begin{equation*}
d(x,y) = \begin{cases}
d_{X_i}(x,y) & \text{ if }x,y\in X_i,\\
|\tfrac{1}{2^{i-1}}-\tfrac{1}{2^{j-1}}| & \text{ if }x\in X_i, y\in X_j, i\neq j,\\
\tfrac{1}{2^{i-1}} & \text{ if }x\in X_i, y=\infty,
\end{cases}
\qquad
f(x)=\begin{cases}
f_i(x) & \text{ if }x\in X_i,\\
\infty & \text{ if }x=\infty.
\end{cases}
\end{equation*}
Note that we are forced to add an extra fixed point $\infty$ to make the resulting space compact. Even though polynomial entropy does not have a countable union rule in general, we do get one for this kind of decreasing topological sum.

\begin{lem}\label{L:top-sum}
	In the above topological sum, $\hpol(f)=\sup_i\hpol(f_i)$.
\end{lem}
\begin{proof}
	Fix $\epsilon>0$ and choose $k$ large enough that $\frac{1}{2^{k}}<\epsilon$. Then the (invariant) sets $X_i$ for $i\geq k+1$ are in the  $\varepsilon$-neighborhood of $\infty$ and so
	\begin{align*}
	\sep(n,\epsilon,f,X) &\leq 1 + \sep(n,\epsilon,f,\bigcup_{i=1}^k X_i) \\
	&\leq 1 + \sum_{i=1}^k \sep(n,\epsilon,f_i,X_i).
	\end{align*}
	Therefore $\hpol^\epsilon(f) \leq \max_{1\leq i\leq k} \hpol^\epsilon(f_i) \leq \max_{1\leq i\leq k} \hpol(f_i)$. Sending $\epsilon$ to zero we get $\hpol(f) \leq \sup_i \hpol(f_i)$. The reverse inequality is trivial, since $(X,f)$ contains each $(X_i,f_i)$ as a subsystem.
\end{proof}

\begin{prop}\label{P:allA}
	For each $a\in[0,\infty]$ there is a homeomorphism $f$ on a compact metric space $X$ such that $\hpol(f)=a$.
\end{prop}
\begin{proof}
	The polynomial entropy of the identity map is 0, and for any $a\in(0,\infty]$ we may find by Proposition~\ref{P:denseA} a sequence of homeomorphisms $(X_i, f_i)$ with $\hpol(f_i) \nearrow a$. Then by Lemma~\ref{L:top-sum} the topological sum has polynomial entropy $a$. But the topological sum of homeomorphisms is clearly again a homeomorphism.
\end{proof}

In the proof of this proposition we have constructed $X$ as a disconnected space (if $a\neq 0$). However, we can prove the result also for continua.

\begin{thm}\label{T:flex-C}
	For each $a\in[0,\infty]$ there is a homeomorphism on a continuum with polynomial entropy $a$.
\end{thm}
\begin{proof}
	Fix $a\in[0,\infty]$. By Proposition~\ref{P:allA}, there is a homeomorphism $f$ on a compact metric space $X$  with $\hpol(f)=a$. If $X$ is a connected space, then it is already a continuum and there is nothing to prove. Otherwise consider the product system $\widehat{f}=f \times \textrm{id}$ on the product space $X \times [0,1]$. By the product rule, $\hpol(\widehat{f})=\hpol(f)+0=a$. Now form the factor space $Y = X\times[0,1] / (X\times\{1\})$ by gluing $X\times\{1\}$ to a single point. This is a well-known space called the cone over $X$, and the map $\widehat{f}$ factors through the quotient map to a continuous map on $Y$, call it $g$. It is easy to see that $g$ is invertible, hence a homeomorphism, and $Y$ is connected, hence a continuum. Since $(Y,g)$ contains $(X,f)$ as a subsystem we have $\hpol(g)\geq a$, and since $g$ is a factor of $\widehat{f}$ we have $\hpol(g)\leq a$.
\end{proof}

%%%%%%%%%%%%%%%%%%%%%%%%%%%%%%%%%%%%%%%%%%%%%%%%%%%%%%%%%%%%%%%%%%%%%%%%%%%%%%%%%%%%%

%%%%%%%%%%%%%%%%%%%%%%%%%%%%%%%%%%%%%%%%%%%%%%%%%%%%%%%%%%%%%%%%%%%%%%%%%%%%%%%%%%%%%

\subsection{Flexibility for dendrite maps}\label{SS:dend}

%%%%%%%%%%%%%%%%%%%%%%%%%%%%%%%%%%%%%%%%%%%%%%%%%%%%%%%%%%%%%%%%%%%%%%%%%%%%%%%%%%%%%

We already know that polynomial entropy of continuous interval maps, if finite, takes only integer values. On the other hand, there are no restrictions for polynomial entropy of homeomorphisms within the class of all continua. In this section we show that even for continuous dendrite maps the polynomial entropy is very flexible and in particular it can take many non-integer values.

We start with a one-sided subshift $(X, \sigma)$ on 2 symbols, $X \subseteq \{0,1\}^{\Nzero}$ with non-integer polynomial entropy; by Proposition~\ref{P:Cass} we may take $\hpol (\sigma)$ to be any real number in the interval $(1,2)$. We follow a universal construction for embedding subshifts into dynamical systems on dendrites due to \cite{KKM}. 
The main idea is to embed the full shift $\left( \{0,1\}^{\Nzero}, \sigma \right)$ into a dynamical system on the Gehman dendrite $\G$ and then to pass to an appropriate subdendrite depending on the considered subshift $(X, \sigma)$.
Recall that $\G$ is the topologically unique dendrite whose endpoint set is homeomorphic to the Cantor set and whose branching points are all of order 3.
In particular, $\G$ contains a root point $c_{\emptyset}$, branch points $c_x$ for each nonempty finite word $x \in \{0, 1\}^{< \infty}$ and endpoints $e_x$ for each infinite word $x \in  \{0, 1\}^{\Nzero}$. For each nonempty finite word $x \in \{0, 1\}^{< \infty}$ we denote by $B_x$ the arc $[c_{\beta(x)}, c_x]$ where $\beta(x_0\dots x_{n-2}x_{n-1}) = x_0\dots x_{n-2}$ and $\beta(0)=\beta(1)=\emptyset$, see Figure~\ref{G}.
We define a metric $d$ on $\G$ such that each $B_x$ is an isometric copy of an interval of length $2^{-\lvert x\rvert}$, where $\left \lvert x_0...x_{n-1} \right\rvert =n$ denotes the {\it length} of a finite word.
As a dendrite $\G$ is uniquely arcwise connected, so we complete the definition of the metric by letting the distance $d(p,q)$ between $p,q\in \G$ be the length of the unique arc in $\G$ with endpoints $p,q$.

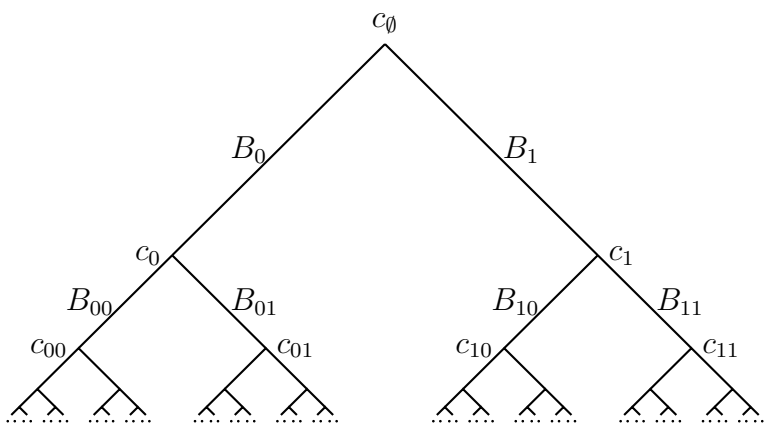
\begin{figure}[htb!!]
	\begin{center}
		\begin{tikzpicture}[
		declare function={
			fx(\x,\i) = (\i==0)*(0.44*\x) + (\i==1)*(10+0.44*(\x-10));
			fy(\y,\i) = 0.44*\y;
		}
		]
		\pgfmathsetmacro{\c}{5};
		\pgfmathsetmacro{\d}{5};
		\pgfmathsetmacro{\a}{fx(\c,0)};
		\pgfmathsetmacro{\b}{fx(\d,0)};
		\pgfmathsetmacro{\e}{fx(\c,1)};
		\pgfmathsetmacro{\f}{fy(\d,1)};
		\draw[thick] (\a,\b) -- (\c,\d) node [midway, left] {$B_0$};
		\draw[thick] (\c,\d) -- (\e,\f) node [midway, right] {$B_1$};
		\path (\c,\d) node [above] {$c_{\emptyset}$};
		\foreach \i in {0,1} {
			\draw[thick] ({fx(\a,\i)},{fy(\b,\i)}) -- ({fx(\c,\i)},{fy(\d,\i)}) node [midway, left] {$B_{\i0}$};
			\draw[thick] ({fx(\c,\i)},{fy(\d,\i)}) -- ({fx(\e,\i)},{fy(\f,\i)}) node [midway, right] {$B_{\i1}$};
			\ifthenelse{\i=0}{\path ({fx(\c,\i)},{fy(\d,\i)}) node [left] {$c_{\i}$}}{\path ({fx(\c,\i)},{fy(\d,\i)}) node [right] {$c_{\i}$}};
		}
		\foreach \i in {0,1} {
			\foreach \j in {0,1} {
				\draw[thick] ({fx(fx(\a,\i),\j)},{fy(fy(\b,\i),\j)}) -- ({fx(fx(\c,\i),\j)},{fy(fy(\d,\i),\j)}) -- ({fx(fx(\e,\i),\j)},{fy(fy(\f,\i),\j)});
				\ifthenelse{\i=0}{\path ({fx(fx(\c,\i),\j)},{fy(fy(\d,\i),\j)}) node [left] {$c_{\j\i}$}}{\path ({fx(fx(\c,\i),\j)},{fy(fy(\d,\i),\j)}) node [right] {$c_{\j\i}$}};
				
		}}
		\foreach \i in {0,1} {
			\foreach \j in {0,1} {
				\foreach \k in {0,1} {
					\draw[thick] ({fx(fx(fx(\a,\i),\j),\k)},{fy(fy(fy(\b,\i),\j),\k)}) -- ({fx(fx(fx(\c,\i),\j),\k)},{fy(fy(fy(\d,\i),\j),\k)}) -- ({fx(fx(fx(\e,\i),\j),\k)},{fy(fy(fy(\f,\i),\j),\k)});
		}}}
		\foreach \i in {0,1} {
			\foreach \j in {0,1} {
				\foreach \k in {0,1} {
					\foreach \l in {0,1} {
						\draw[thick] ({fx(fx(fx(fx(\a,\i),\j),\k),\l)},{fy(fy(fy(fy(\b,\i),\j),\k),\l)}) -- ({fx(fx(fx(fx(\c,\i),\j),\k),\l)},{fy(fy(fy(fy(\d,\i),\j),\k),\l)}) -- ({fx(fx(fx(fx(\e,\i),\j),\k),\l)},{fy(fy(fy(fy(\f,\i),\j),\k),\l)});
		}}}}
		\foreach \i in {0,1} {
			\foreach \j in {0,1} {
				\foreach \k in {0,1} {
					\foreach \l in {0,1} {
						\foreach \m in {0,1} {
							\draw[fill=black] ({fx(fx(fx(fx(fx(\a,\i),\j),\k),\l),\m)},{0}) circle [radius=0.01];
							% \draw[fill=black] ({fx(fx(fx(fx(fx(\c,\i),\j),\k),\l),\m)},{0}) circle [radius=0.01];
							\draw[fill=black] ({fx(fx(fx(fx(fx(\e,\i),\j),\k),\l),\m)},{0}) circle [radius=0.01];
		}}}}}
		\end{tikzpicture}
	\end{center}
	\caption{Gehman dendrite $\G$}\label{G}
\end{figure}
%
%
%%%%%%%%%%%%
%%%%%%%%%%%%
%
% 
The dynamics of the full shift are realized on the endpoint set of $\G$ by a continuous map $F: \G \to  \G $ for which $F(e_x) = e_{\sigma(x)}$ and
$F(c_x)=c_{\sigma(x)}$, where $\sigma (\emptyset)  = \emptyset$, and extending $F$ linearly on each $B_x$. In particular $F(B_0)=F(B_1)=\{c_{\emptyset}\}$.
Further recall that for each subshift (closed, nonempty, $\sigma$-invariant subset) $X \subseteq  \{0, 1\}^{\Nzero}$, the corresponding subdendrite $\G_X$ formed by taking the union in $\G$ of the arcs joining $c_{\emptyset}$ to the points $e_x, x \in X$, is invariant under $F$.
Note that $c_x \in \G_X$ if and only if the finite word $x$ occurs as a subword in some element of the shift space $X$. The set of all such finite words is denoted $\Lg(x)$ and is called the {\it language} of $X$.
We write $F_X$ for the restriction of $F$ to $\G_X$.\footnote{Recall that the metric usually used for the shift space $X$ is $\rho(x,y)=2^{-\inf\{i~:~x_i\neq y_i\}}$. The restriction of our metric $d$ to the endpoint set of $\G_X$ gives distances twice as large $d(e_x,e_y)=2\rho(x,y)$. This in particular means that the subshift $(X,\sigma|_X)$ and the system given by the restriction of $F_X$ to the endpoint set of $\G_X$ are topologically conjugate and so have the same polynomial entropy (whose value is given by Lemma~\ref{LemWords}). }

Note that $\hpol\left( F_X|_{\textnormal{End}(\G_X)} \right)= \hpol\left( \sigma|_X \right)$ because of the conjugacy.

In our construction of dendrite maps with non-integer polynomial entropy we will use the following interval map.

\begin{lem}\label{L:map g}
	For the surjective map $g\colon [0,1] \to [0,1]$ defined by 	
	\begin{equation*}
	g(x) =  \left \{   \begin{array}{l  l} \vspace{0.8em}
	2x, & \mbox{  if  $\ x \in \left[0, \frac12 \right]$, }\\ 
	1,  & \mbox{  if   $\ x \in \left[ \frac12 , 1 \right] $}  \\ 
	\end{array} \right.
	\end{equation*}
	we have $\hpol(g)=1$.
\end{lem}

\begin{proof}
This immediately follows from Corollary~\ref{C:monot}.
\end{proof}

\begin{prop}\label{P:plus1}
	For every subshift $(X, \sigma), X \subseteq  \{0, 1\}^{\Nzero} $ we have 
	$${ \hpol(F_X) = \hpol(\sigma|_X) + 1.}$$
\end{prop}

\begin{rem}
	Contrary to polynomial entropy, the topological entropy does not change when we extend any subshift $X$ to the corresponding dendrite map, i.e.\ $\htop(F_X)=\htop(\sigma_X)$.
\end{rem}

\begin{proof}
	In addition to $F_X: \G_X \to \G_X$ we consider the interval map $g: [0,1]\to [0,1]$ from Lemma~\ref{L:map g} and we show that $g$ is a factor of $F_X$. We consider also a product map $P: X \times  [0,1]\to X \times  [0,1]$ which is an extension of $F_X$. 
	
	Let $\pi_1: \G_X \to [0,1]$ be defined by 
	\begin{equation*}
	\begin{array}{l  l} \vspace{0.8em}  
	\pi_1(e_x)  = 0 & \mbox{for all } x \in X,\\
	\pi_1(c_x) = 2^{-|x|} & \mbox{for all } x \in \Lg(X).\\
	\end{array}
	\end{equation*}
	For every arc $B_x, x \in \Lg(X) \setminus \{\emptyset\} $, we already know the $\pi$-images of the endpoints, and we extend $\pi$ to the rest of the arc in a linear fashion. 
	Note that $\pi_1: \G_X \to [0,1]$ is a continuous surjection.
	Observe that the diagram
	\begin{equation*}
	\begin{tikzcd}[row sep=1cm, column sep=1.5cm]
	\G_X \arrow[r, "F_X"] \arrow[d, shift left=-1, "\pi_1" swap] & 
	\G_X \arrow[d, shift left=-1, "\pi_1"] \\
	\left[0,1\right] \arrow[r, "g" swap] & \left[0,1\right] \end{tikzcd}
	\end{equation*}
	commutes at the endpoints of $\G_X$ and at the endpoints of all arcs $B_x$. 
	Since $\pi_1, F_X$ are linear on $B_x$ and $g$ is linear on $\pi_1(B_x)$ (either constant or with slope 2), we see that the diagram commutes everywhere on $G_X$, so $g$ is a factor of $F_X$.
	
	Now we define the product map $P: X\times [0,1] \to X\times [0,1]$, by $P=\sigma|_{X} \times g$, i.e.,
	$$P(x, y) = \left(\sigma(x), g(y) \right).$$ 
	This is continuous.
	Let $\pi_2: X\times [0,1] \to \G_X$ be defined for 
	$x=x_0 x_1 \cdots \in X$ as follows
	\begin{equation*}
	\begin{array}{l  l} \vspace{0.5em}  
	\pi_2(x,0)  = e_x & \\ \vspace{0.5em}  
	\pi_2(x,1) =c_{\emptyset } & \\
	\pi_2(x, 2^{-n}) = c_{x_0 \cdots x_{n-1}}.\\
	\end{array}
	\end{equation*}
	
	Then for every arc $\{x\} \times \left[ 2^{-n}, 2^{-n+1} \right]$ the $\pi_2$-images of its endpoints are the endpoints of the arc $B_{x_0 x_1\cdots x_{n-1}}$ and we may extend $\pi_2$ linearly. %%
	
	Note that $\pi_2: X\times [0,1] \to \G_X$ is a continuous surjection, and similarly as before the following diagram commutes
	\begin{equation*}
	\begin{tikzcd}[row sep=1cm, column sep=1.5cm]
	X\times \left[0,1\right]  \arrow[r, "P"] \arrow[d, shift left=-1, "\pi_2" swap] & 
	X\times \left[0,1\right]   \arrow[d, shift left=-1, "\pi_2"] \\
	\G_X \arrow[r, "F_X", swap] & 
	\G_X
	\end{tikzcd}
	\end{equation*}
	so that $P$ is an extension of $F_X$.
	
	Since $\hpol(g)=1$ by Lemma~\ref{L:map g}, now we are ready to give an upper bound for $\hpol(F_X)$. Since $F_X$ is a factor of $P$ and $P= \sigma|_X \times g$ is a product, we know from Subsection~\ref{SS:prop} that 
	$$ \hpol(F_X) \leq \hpol(P) = \hpol(\sigma|_X) +\hpol(g) = \hpol(\sigma|_X) +1.$$ 
	
	Now we turn our attention to estimating $\hpol(F_X)$ from below.
	We construct $(n, \epsilon, F_X)$-separated sets which surprisingly (at least at first sight) do not contain any endpoints of $\G_X$ but are still large enough for our purposes.
	We claim that for $n \in \Nzero$ and $\epsilon < \frac12$, the set $S=\left\{c_x : x \in \Lg(X), \lvert x\rvert \leq n \right\}$ is $(n, \epsilon, F_X)$-separated.
	From the definition of the metric $d$, it is clear that $c_{\emptyset}$ is further than $\epsilon$ from any other point in $S$. 
	Therefore if $c_x, c_y \in S$ have indices with different lengths $i= \lvert x \rvert < j = \lvert y \rvert  \leq n$, then after $i$ iterates $d\left(F^i_X(c_x),  F^i_X(c_y) \right) = d\left(c_{\emptyset}, c_{\sigma^i(y)} \right) \geq 1/2>\epsilon.$
	Finally, if distinct points $c_x, c_y \in S$ have indices  of the same length $k \leq n$, then we write $x=x_0 \cdots x_{k-1}, y= y_0 \cdots y_{k-1}$. 
	Then there exists $i < k$ such that $x_i \neq  y_i$.
	Thus $\sigma^i(x), \sigma^i(y)$ start with different symbols and $F^i_X(c_x)=c_{\sigma^i(x)}, F^i_X(c_y)=c_{\sigma^i(y)}$. 
	Since one of $\sigma^i(x), \sigma^i(y)$ starts with the symbol $0$ and the other starts with the symbol $1$, we get $d\left(F^i_X(c_x),  F^i_X(c_y) \right) \ge d(c_0, c_1) = 1 > \epsilon.$
	Thus we have shown that $\sep(n, \epsilon, F_X) \ge \#S$ and $\#S$ is simply $\omega(0) + \omega(1) + \cdots + \omega(n),$ where $\omega(i)$ counts the number of words of length $i$ in the language of $X$.
	To simplify the calculations we consider this inequality only for even numbers:
	\begin{equation}\label{sep}
	\sep(2n, \epsilon, F_X) \ge \omega(0) + \cdots + \omega(n)+ \omega(n+1)+\cdots+ \omega(2n), \quad \epsilon < \dfrac12, n\in\Nzero.
	\end{equation}
	Since the complexity function of a subshift is non-decreasing, we get $\sep(2n, \epsilon, F_X) \ge n \omega(n)$. 
	This suffices to give the desired lower bound.
	Indeed for every $\epsilon < \frac12$ we have
	\begin{equation*}
	\begin{aligned}
	\hpol^{\varepsilon}(F_X)  & = \limsup_{n\to \infty}\dfrac{\log \sep(n, \epsilon, F_X)}{\log n}  \ge \limsup_{n\to \infty}\dfrac{\log \sep(2n, \epsilon, F_X)}{\log(2n)} \\
	& \ge \limsup_{n\to \infty}\dfrac{\log(n \omega(n))}{\log(2n)}  \\
	& = \limsup_{n\to \infty} \left ( \dfrac{\log n}{\log 2+\log n} + \dfrac{\log \omega(n)}{\log 2+\log n} \right ) \\
	& =1+\hpol(\sigma|_X),
	\end{aligned}
	\end{equation*}
	where we have used Lemma~\ref{LemWords} to evaluate the limes superior.
	Now sending $\epsilon \to 0$ we get 
	\begin{equation*}
	\hpol(F_X) \ge 1+\hpol(\sigma|X).\qedhere
	\end{equation*}
\end{proof}

\begin{thm}\label{T:flex-D}
	For every real number $\alpha \in (2,3)$ there is a dendrite $D$ and a continuous surjective map $f\colon D\to D$ with $\hpol(f)=\alpha$. 
\end{thm}

\begin{proof}
	Fix $\alpha \in (2,3)$. Then $\alpha -1 \in (1,2)$ and by Proposition~\ref{P:Cass} there is a one-sided surjective subshift $(X,\sigma)$ on 2 symbols with $\hpol(\sigma) = \alpha -1$. By Proposition~\ref{P:plus1}, the subdendrite $D=\G_X$ of the Gehman dendrite and the corresponding map $f= F_X$ on $D$ are such that $\hpol(f)=\alpha$. Moreover, since the subshift is surjective, also $f$ is obviously surjective.
\end{proof}

\begin{rem}\label{R:dend}
More can be said. Using Proposition~\ref{P:Kurka}, rather than Proposition~\ref{P:Cass}, and using a Gehmann-like dendrite with branch points of higher order, rather than the Gehman dendrite, we can get dendrite maps with polynomial entropy taking arbitrary values in $[2,\infty]$. The values $0, 1$ can also be attained trivially (since the interval is a dendrite).  However, even then the question is left open whether polynomial entropy can take non-integer values less than 2 for continuous maps on dendrites. Note that~\cite{HL19} left open the same question for homeomorphisms on $\mathbb{S}^2$.
\end{rem}

\begin{Q}
Is there a dendrite $X$ and a continuous map $f:X\to X$ such that $\hpol(f)\in(0,1)\cup(1,2)$? Is every positive real number the polynomial entropy of a dendrite map?
\end{Q}

%%%
%%%====     Acknowledgments     ====
%%%
\section*{Acknowledgments} \noindent 
The first and second authors were supported by the Czech Republic RVO funding for I\v{C}47813059.
The third author was supported by the Slovak Research and Development Agency under contract No.~APVV-15-0439 and by VEGA grant 1/0158/20.
Much of this paper was written during the stay of the first and second authors in Bansk\'{a} Bystrica; they are grateful for the hospitality of Matej Bel University. The first author's stay was supported by the Silesian University in Opava and the European Union through project CZ.02.2.69/0.0/0.0/18\_053/0017871.

%%%%%%%%%%%%%%%%%%%%
%%%====         Literatura      =======
%%%%%%%%%%%%%%%%%%%%

\end{document}